\documentclass[12pt]{article}
\usepackage{amsmath,amssymb,amsthm,a4wide,color,graphicx}
\newtheorem{theorem}{Theorem}[section]
\newtheorem{lemma}{Lemma}[section]

\definecolor{grey}{rgb}{0.6, 0.6, 0.6}
\newcommand{\skippedterms}[1]{\underbrace{#1}}

\newcommand{\kaptil}{\tilde{\kappa}}
\newcommand{\uldeta}{\underline{d\eta}}
\newcommand{\uldph}{\underline{d\hat{p}}}
\newcommand{\uldp}{\underline{dp}}
\begin{document}
\title{Nonlinearity parameter imaging in the frequency domain}
\author{Barbara Kaltenbacher\footnote{
Department of Mathematics,
Alpen-Adria-Universit\"at Klagenfurt.
barbara.kaltenbacher@aau.at.}
\and
William Rundell\footnote{
Department of Mathematics,
Texas A\&M University,
Texas 77843. 
rundell@math.tamu.edu}
}
\date{\vskip-3ex}
\maketitle  
\vspace{-10pt}
\begin{abstract}
Nonlinearity parameter tomography leads to the problem of identifying a coefficient in a nonlinear wave equation (such as the Westervelt equation) modeling ultrasound propagation. In this paper we transfer this into frequency domain, where the Westervelt equation gets replaced by a coupled system of Helmholtz equations with quadratic nonlinearities. For the case of the to-be-determined nonlinearity coefficient being a characteristic function of an unknown, not necessarily connected domain $D$, we devise and test a reconstruction algorithm based on weighted point source approximations combined with Newton's method. 
In a more abstract setting, convergence of a regularised Newton type method for this inverse problem is proven by verifying a range invariance condition of the forward operator and establishing injectivity of its linearisation.
\end{abstract}
\textbf{key words:} nonlinearity parameter tomography, multiharmonic expansion, Westervelt equation, Helmholtz equation, extended sources, point sources, Newton's method, range invariance condition

\section{Introduction}
Nonlinearity parameter tomography 
\cite{Bjorno1986, BurovGurinovichRudenkoTagunov1994, Cain1986, 
IchidaSatoLinzer1983, PanfilovaSlounWijkstraSapozhnikovMischi:2021, VarrayBassetTortoliCachard2011,ZhangChenGong2001, ZhangChenYe1996}, 
is a technique for enhancing ultrasound imaging and amounts to identifying the spatially varying coefficient $\eta=\eta(x)$ 
in the Westervelt equation
\begin{equation}\label{Westervelt}
p_{tt}-c^2\triangle p - b\triangle p_t = \eta(p^2)_{tt} + h \mbox{ in }(0,T)\times\Omega\,, \quad
\end{equation}
where $p$ is the acoustic pressure, $c$ the speed of sound, $b$ the diffusivity of sound, and $h$ the excitation, from observations of the pressure 
\begin{equation}\label{observation_time}
y(x,t) = p(x,t), \quad(x,t)\in\Sigma\times(0,T).
\end{equation}
on some manifold $\Sigma$ immersed in the acoustic domain $\Omega$ or attached to its boundary $\Sigma\in\overline{\Omega}$; see \cite{AcostaUhlmannZhai:2022, BB9, BB10, BB15} and the references therein.

While uniqueness from the Dirichlet-to-Neumann operator has been established in \cite{AcostaUhlmannZhai:2022}, our aim here is to reconstruct $\eta$ from the single boundary measurement \eqref{observation_time} like in \cite{BB9, BB10, BB15}.

\medskip

Here we will consider this problem in the frequency domain, inspired by the concept of harmonic imaging \cite{AnvariForsbergSamir15,Uppal2010,VarrayBassetTortoliCachard2011}. Due to the quadratic nonlinearity appearing in the PDE, this is not directly possible by the usual approach of taking the Fourier transform in time. Rather, the idea is to use a multiharmonic ansatz \cite{periodicWestervelt} as follows.

Assuming periodic excitations of the specific form 
$h(x,t)=\Re(\hat{h}(x)e^{\imath \omega t})$ 
for some fixed frequency $\omega$ and $\hat{h}\in L^2(\Omega;\mathbb{C})$ 
and inserting a multiharmonic expansion for a time periodic solution of \eqref{Westervelt} (that due to periodicity of $h$ can be proven to exist and be unique)
$p(x,t)
= \Re\left(\sum_{k=1}^\infty \hat{p}_k(x) e^{\imath k \omega t}\right)$
into \eqref{Westervelt},
yields the infinite system of coupled linear Helmholtz type PDEs 
\begin{equation}\label{multiharmonic_nonlinear}
\begin{aligned}
&m=1:&& -\omega^2 \hat{p}_1-(c^2+\imath\omega b) \triangle \hat{p}_1 = 
\hat{h}
\skippedterms{ \ 
-\frac{\eta}{2}\omega^2
\sum_{k=3:2}^\infty\overline{\hat{p}_{\frac{k-1}{2}}} \hat{p}_{\frac{k+1}{2}}
}
\\
&m\in \{2,\ldots,M\}:&& -\omega^2 m^2 \hat{p}_m-(c^2+\imath\omega m b) \triangle \hat{p}_m \\&&&\hspace*{3cm}= -\frac{\eta}{4}\omega^2 m^2 \Bigl(\sum_{\ell=1}^{m-1} \hat{p}_\ell \hat{p}_{m-\ell} 
\skippedterms{\ 
+ 2\sum_{k=m+2:2}^\infty\overline{\hat{p}_{\frac{k-m}{2}}} \hat{p}_{\frac{k+m}{2}}
}
\Bigr).
\end{aligned}
\end{equation}
The equivalence of \eqref{multiharmonic_nonlinear} to \eqref{Westervelt} holds with $M=\infty$, as shown in \cite{periodicWestervelt}. 
The fact that in place of a single Helmholtz equation we have a system (in theory even an infinite one) reveals that nonlinearity actually helps the identifiability. This can be explained by the additional information available due to the appearance of several higher harmonics (similarly to several components arising in the asymptotic expansion in \cite{KurylevLassasUhlmann2018}).
In practice, the underbraced terms are often skipped and the expansion is only considered up to $M=2$ or $M=3$. This is due to the fact that the strength of the signal in these higher harmonics decreases extremely quickly. In fact in our reconstructions, only two of them will be of effective use as the third harmonic  only provides marginal improvement over the second one.

In our reconstructions in Section~\ref{sec:reconstructions}, we will focus on the case of a piecewise constant coefficient $\eta=\eta_0\chi_D$ with a known constant $\eta_0$ and an unknown domain $D$, so that \eqref{multiharmonic_nonlinear} (upon skipping the grey terms) becomes
\begin{equation}\label{multharm_obj}
\begin{aligned}
&m=1:&& \triangle \hat{p}_1 +\kappa^2\hat{p}_1= \hat{h}
\\
&m\in \{2,\ldots,M\}:&& \triangle \hat{p}_m + m^2\kappa^2\hat{p}_m=\frac{\eta_0}{4} \, \chi_D\,
m^2\kappa^2 \Bigl(\sum_{\ell=1}^{m-1} \hat{p}_\ell \hat{p}_{m-\ell}, 
\Bigr)
\end{aligned}
\end{equation}
where $\kappa=\frac{\omega}{\sqrt{c^2+\imath\omega b}}$ is the wave number.
We do so for practical relevance (e.g., location of contrast agents such as microbubbles on a homogeneous background)
and for expected better identifiability as compared to a general function $\eta$
(although counterexemples to uniqueness still exist cf., e.g., \cite{AlvesMartinsRoberty:2009,KressRundell:2013}, for the Helmholtz equation as opposed to the Laplace equation).
Typically, $D$ will not necessarily be connected but consist of a union of connected components $D=\bigcup_{\ell=1}^{m} D_\ell$ that we will call inclusions or objects for obvious reasons.

Moreover, throughout this paper we assume the sound speed $c$ to be known and constant. For results (in the time domain formulation \eqref{Westervelt}) on simultaneous identification of space dependent functions $c$ and $\eta$, we refer to \cite{BB15}.    

We will consider \eqref{multiharmonic_nonlinear} on a smooth bounded domain $\Omega\subseteq\mathbb{R}^d$, $d\in\{2,3\}$ with observations on a subset of $\partial\Omega$ and equip it with a boundary damping condition 
\begin{equation}\label{impedance} 
\partial_\nu \hat{p}_m +(\imath m\omega\beta + \gamma) \hat{p}_m  = 0 \quad\mbox{ on }\partial\Omega
\end{equation}
with $\beta,\,\gamma\geq0$. These are direct translations to frequancy domain of zero and first order absorbing boundary conditions in time domain, see, e.g., the review
articles~\cite{GivoliBookChapter08,Hagstrom1999} and the references therein. Indeed, these boundary attenuation conditions even allow us to skip the interior damping and assume $\kappa$ to be real valued, as has been shown in \cite{bndystabWestervelt} in the time domain setting of \eqref{Westervelt}. We will do so by working with a real valued wave number $\kaptil$ in the numerical tests of Section~\ref{sec:reconstructions}.

In the case where the observation manifold is contained in the boundary of the domain $\Omega$, we can choose between writing the data \eqref{observation_time} as Dirichlet trace or, via the impedance condition \eqref{impedance}, with $g_m=-(\imath m\kappa+\gamma) y_m$, as Neumann trace 
\begin{equation}\label{observation}
y_m = \hat{p}_m\ \textup{ or }\quad 
g_m = \partial_\nu \hat{p}_m \quad\mbox{ in }\Sigma, \quad m\in\{2,\ldots,M\}.
\end{equation}
In our numerical reconstructions we will also consider the practically relevant case of only partial data being available with $\Sigma \subseteq\partial\Omega$ being a strict subset. 
Note that according to the first line in \eqref{multharm_obj}, that does not contain the unknown $D$, observations of the fundamental harmonic $y_1$ or $g_1$ are not expected to carry essential information on $D$ and are therefore neglected.

\section{A reconstruction method for piecewise constant $\eta$ and numerical results} \label{sec:reconstructions}

We first of all consider \eqref{multharm_obj} for $M=2$ and devise a reconstruction method, based on the approach in \cite{KressRundell:2013}.
While the algorithms described below work in both 2-d and 3-d, we confine the exposition and our numerical experiments to two space dimensions.
In our numerical tests we will also study the question of whether taking into account another harmonic $M=3$ improves the results.

Having computed $\hat{p}_1$ from the first equation in \eqref{multiharmonic_nonlinear} with given excitation $\hat{h}$, the problem of determining $\eta$ from the second equation in \eqref{multiharmonic_nonlinear} reduces to an inverse source problem for the Helmholtz equation
\begin{equation}\label{Helmholtz}
\triangle u +\kaptil^2 u = \kaptil^2 \eta \, \tilde{f} \quad\mbox{ in }\Omega
\end{equation}
where $u=\hat{p}_2$, $\kaptil=\frac{2\omega}{c}$, $\tilde{f}=\frac{1}{4c^2}\hat{p}_1^2$. 

In the case of a piecewise constant coefficient as considered here, \eqref{Helmholtz} becomes
\begin{equation}\label{Helmholtz_obj}
\triangle u +\kaptil^2 u = \kaptil^2 \, \chi_D \, f \quad\mbox{ in }\Omega.
\end{equation}
with $f=\eta_0\tilde{f}$.
There exists a large body of work on inverse source problems for the Helmholtz equation.  
Two particular examples for the case of extended sources as related to our setting are \cite{Ikehata:1999,KressRundell:2013}. 
We also point to, e.g., \cite{AcostaChowTaylillamazar:2012,AlvesMartinsRoberty:2009,BaoLinTriki:2010,ChengIsakovLu:2016,EllerValdivia:2009} for inverse source problems with multi frequency data; however these do not cover the important special case of restricting observations to higher harmonics of a single fundamental frequency. 

We here intend to follow the approach from \cite{KressRundell:2013}.
Like there, as an auxiliary problem, we will consider the Helmholtz equation with point sources
\begin{equation}\label{Helmholtz_pts}
\triangle u +\kaptil^2 u = \sum_{k=1}^n \lambda_k \delta_{S_k} \quad\mbox{ in }\Omega.
\end{equation}
with $\delta$ distributions located at points $S_k$,
or more generally with a measure $\mu\in\mathcal{M}(\Omega)=C_b(\overline{\Omega})^*$ as right hand side
\begin{equation}\label{Helmholtz_meas}
\triangle u +\kaptil^2 u = \mu \quad\mbox{ in }\Omega.
\end{equation}
The PDEs \eqref{Helmholtz_obj}, \eqref{Helmholtz_pts}, \eqref{Helmholtz_meas} are equipped with impedance boundary conditions
\begin{equation}\label{imped} 
\partial_\nu u + \imath \kaptil u = 0 \quad\mbox{ on }\partial\Omega.
\end{equation}
Results on well-posedness of the forward problems \eqref{Helmholtz}, \eqref{imped} and  \eqref{Helmholtz_pts}, \eqref{imped} can be found, e.g., in \cite[Section VIII]{Melenk:1995} and \cite[Section 2]{PTTW:2020}. 

\medskip

An essential fact connecting \eqref{Helmholtz_obj} and \eqref{Helmholtz_pts} is that for any solution $w$ of the homogeneous Helmholtz equation $\triangle w +\kaptil^2 w=0$ on $\Omega$, from Green's second identity, written in the form
\[
\int_\Omega \Bigl(u\, (\triangle w +\kaptil^2 w) - w\, (\triangle u +\kaptil^2 u)\Bigr)\, dx
= \int_{\partial\Omega} \Bigl(u\, (\partial_\nu w+\imath\kaptil w) - w\, (\partial_\nu u+\imath\kaptil u) \Bigr)\, ds
\]
the following relations hold
\begin{equation}\label{moment_Helmholtz}
\begin{aligned}
&\int_{\partial\Omega}\partial_\nu u\, (\partial_\nu w+\imath\kaptil w)\, ds\\[-1ex] 
&= -\imath\kaptil\int_{\partial\Omega} u\, (\partial_\nu w+\imath\kaptil w)\, ds 
= \begin{cases}
\imath\kaptil\int_D \kaptil^2\, f\, w\, dx&\mbox{ for \eqref{Helmholtz_obj}, \eqref{impedance}}\\
\imath\kaptil \sum_{k=1}^n \lambda_k w(S_k) &\mbox{ for \eqref{Helmholtz_pts}, \eqref{impedance}.}
\end{cases}
\end{aligned}
\end{equation}
Combining this with a mean value identity for the Helmholtz equation
\begin{equation}\label{meanvalue_Helmholtz}
\frac{1}{|B_r(x_0)|}\int_{B_r(x_0)} w \, dx = \Gamma(\tfrac{d}{2}+1) \frac{J_{d/2}(\kaptil\, r)}{(\kaptil\, r/2)^{d/2}} \, w(x_0)
\end{equation}
for any $r>0$, and $x_0\in\Omega$ such that $B_r(x_0)\subseteq\Omega$, and $w$ solving $\triangle w +\kaptil^2 w=0$ (see, e.g., \cite{Kuznetsov:2021} and the references therein), equivalence of \eqref{Helmholtz_obj}, \eqref{Helmholtz_pts} in the case of constant background $f$ is obtained. 

\begin{lemma} \label{lem:discs}
Assume that $D$ can be represented as the union of finitely many disjoint discs or balls.
Then the flux moments $\int_{\partial\Omega}\partial_\nu u\, (\partial_\nu w+\imath\kaptil w)\, ds$ (for $w$ in the kernel of $\triangle w +\kaptil^2\textup{id}$) of $u$ solving the Helmholtz equation \eqref{Helmholtz_obj}, \eqref{imped} with $f\equiv \textup{const.}$ coincide with the flux moments resulting from finitely many weighted point sources \eqref{Helmholtz_pts}, \eqref{imped}.
\end{lemma}


The method from \cite{KressRundell:2013} uses a Pad\'e approximation scheme (see \cite{HankeRundell:2011}, which was inspired by \cite{ElBadiaHaDuong:2000}) for recovering point sources in the Laplace equation and a fixed point scheme to extend this for finding point sources in the Helmhotz equation \eqref{Helmholtz_pts}. 
This is proven to converge in \cite[Theorem 1]{KressRundell:2013} for sufficiently small wave numbers $\kaptil$ and the numerical experiments there show that it works exceedingly well for $\kaptil\leq 1$.
However, in ultrasonics, $\kaptil$ is large. Transition from the Laplace point source problem 
to the Helmholtz point source problem 
therefore does not seem to be feasible in that situation. 
However, transition from the Helmholtz point source problem \eqref{Helmholtz_pts} the Helmholtz inclusion problem \eqref{Helmholtz_obj} is still justified by Lemma ~\ref{lem:discs}, in case of circular or spherical inclusions and a constant background $f$.

In place of the Pad\'e approximation algorithm in \cite{KressRundell:2013}, we employ the primal-dual active point PDAP algorithm from \cite{BrediesPikkarainen:2013,PTTW:2020}, which we provide here, for the convenience of the reader. It uses the forward operator $F:\mathcal{M}(\Omega)\to L^2(\Sigma)$, $\mu\mapsto \partial_\nu u\vert_\Sigma$, 
\footnote{$L^2(\Sigma)$ regularity of the flux (in spite of the low $W^{1,q}(\Omega)$, $q<\frac{d}{d-1}$ regularity of $u$) is obtained by bootstrapping from the homogeneous impedance conditions in case of $\Sigma\subseteq \partial \Omega$; otherwise, an assumption of the source domain to be at distance from $\Sigma$ needs to be imposed in order to be able to invoke interior elliptic regularity.}
where $u$ solves \eqref{Helmholtz_meas}, \eqref{imped} and its Banach space adjoint $F^*$.

\smallskip

\noindent
{\bf Algorithm PDAP:}\\[1ex]
For $i=1,2,3,\ldots$
\begin{enumerate}
\item Compute $\xi^i:=F^*(F\mu^i-g)$; determine $\hat{S}^i\in\textup{argmax}_{x\in\Omega}|\xi(x)|$
\item Set $(S^i_1,\ldots,S^i_n):=\textup{supp}(\mu^i)\cup\{\hat{S}^i\}$; 
\item compute a minimizer $\vec{\lambda}^i\in\mathbb{R}^n$ of $j(\vec{\lambda}):=\|F\sum_{k=1}^n\lambda_k\delta_{S^i_k} -g\|^2$  
\item Set $\mu^{i+1}=\sum_{k=1}^n \lambda^i_k \delta_{S^i_k}$
\end{enumerate}

Combining this with the other elements from the method in \cite{KressRundell:2013}, we arrive at the following scheme in case of constant background.

\smallskip

\noindent
{\bf Algorithm 0:}\\[1ex]
Given boundary flux 
$g = g_D =\sum_{\ell=1}^m g_{D_\ell} $
arising from the $m$ unknown objects $D_\ell$ (each of which is the union of $n_\ell$ discs)
with constant background $f$.
\begin{enumerate}
\item[(1.)] 
Identify 
$n = \sum_{\ell=1}^m n_\ell \geq m$
equivalent point sources $S_k$ and weights $\lambda_k$ according to Lemma~\ref{lem:discs} using Algorithm PDAP.\\ 
This also yields a decomposition
$g = g_D =\sum_{k=1}^n g_{pts_k}$ 
of the given data;
\item[(2.)] 
Determine the radii of equivalent discs from weights $\lambda_k$ via the mean value property \eqref{meanvalue_Helmholtz}.\\
Merge these discs into $m$ objects: two discs belong to the same object if their intersection is nonempty;\\ 
Assigning discs and therewith equivalent point sources to objects 
$g_{pts_k} \to g_{pts_{\ell,j}} $
for 
$k\in\{1,\ldots,n\}$, 
$\ell\in\{1,\ldots,m\}$, 
$j\in\{1,\ldots,n_\ell\}$,
also yields a decomposition of the given data:
$g = g_D =\sum_{\ell=1}^m g_\ell$, 
where 
$g_\ell = \sum_{j=1}^{n_\ell} g_{pts_{\ell,j}}$.
\item[(3.)]
For each object 
$D_\ell$, $\ell\in\{1,\ldots,m\}$,
separately, determine the object boundary parametrised by a curve $q_\ell$ from moment matching \eqref{moment_Helmholtz} of data $g_\ell$, using a Newton iteration;\\
\end{enumerate}

As a starting value for each curve $q_\ell$ in (3.) we use the disc with the centroid of the union of discs belonging to the $\ell$-th object as a center and the radius corresponding to the sum of weights within the $\ell$-th object via \eqref{meanvalue_Helmholtz}.
Alternatively to (3.), one could use algorithms from computational geometry for determining the boundary of a union of discs, see, e.g., \cite{EzraHalperinSharir:2002,GoodmanPachPollack:2008} and the citing literature.

\medskip

In case of variable background $f$ as relevant here, cf.  \eqref{multharm_obj}, and/or a set $D$ that is not a finite union of discs, the representation by equivalent discs is not exact and therefore the decomposition of the data according to objects is not valid any more. We therefore replace (3.) by a simultaneous Newton based matching of the flux data $g$ (not of its moments) to the flux data computed from forward simulations according to the collection of parametrised object boundaries.    
We can still regard the discs obtained by (2.) as good starting guesses for Newton's method and thus proceed as follows.

\smallskip

\noindent
{\bf Algorithm 1:}\\[1ex]
given boundary flux 
$g = g_D =\sum_{\ell=1}^m g_{D_\ell} $
arising from the $m$ unknown objects $D_\ell$ 
\begin{enumerate}
\item[(1.)] 
Identify 
$n = \sum_{\ell=1}^m n_\ell \geq m$
approximately equivalent point sources $S_k$ and weights $\lambda_k$ by Algorithm PDAP;
\item[(2.)] 
Determine disc radii from weights $\lambda_k$ via the mean value property \eqref{meanvalue_Helmholtz}.\\
Merge discs to $m$ objects: two discs belong to the same object if their intersection is nonempty;
\item[(3.)]
For all objects $D_\ell$, $\ell\in\{1,\ldots,m\}$, simultaneously, determine the object boundaries parametrised by curves $q_\ell$ by matching the combined observational data \eqref{observation}, using a Newton iteration.
\end{enumerate}

The choice of a starting value for $q_\ell$ in (3.) is the same as in Algorithm 0, namely a disc with center determined as centroid of all discs pertaining to the $\ell$-th object and radius determined by using the sum of weights in \eqref{meanvalue_Helmholtz}.

\subsection{Reconstructions}

Our forward solvers for \eqref{Helmholtz}, \eqref{imped} (in the special cases \eqref{Helmholtz_obj},\eqref{Helmholtz_pts}, of \eqref{Helmholtz}) rely on the fact that with the fundamental solution to the Helmhotz equation
$\mathcal{G}(x)=\frac{\imath}{4} H_0^1(\kaptil|x|)$ in two space dimensions, the solution to
\[
\triangle u^{\mathbb{R}^2} +\kaptil^2 u^{\mathbb{R}^2} = f \quad\mbox{ in }\mathbb{R}^2
\] 
can be determined by convolution $u^{\mathbb{R}^2}=\mathcal{G}*f$.
It thus remains to solve the homogeneous boundary value problem   
\[
\triangle u^{\textup{d}} +\kaptil^2 u^{\textup{d}} = 0 \quad\mbox{ in }\Omega\,, \quad
\partial_\nu u^{\textup{d}} +\imath\kaptil u^{\textup{d}} = g
\] 
with $g=-\partial_\nu u^{\mathbb{R}^2} -\imath\kaptil u^{\mathbb{R}^2}$, which we do by the integral equation approach described in \cite[Sections 3.1, 3.4]{ColtonKress:2013}, that easily extends to the case of impedance boundary conditions.
The solution to \eqref{Helmholtz}, \eqref{imped} is then obtained as $u=u^{\mathbb{R}^2}+u^{\textup{d}}$.
We point to the fact that solving the Helmholtz equation with large wave numbers is a challenging task and a highly active field of research, see, e.g., \cite{LafontaineSpenceWunsch:2022,MelenkSauterTorres:2020,PeterseimVerfuerth:2020} and the references therein.
Since our emphasis lies on a proof of concept for parameter identification, we did not implement any of these high frequency solvers here.

In all our reconstructions it is apparent that the point source reconstruction algorithm from \cite{BrediesPikkarainen:2013,PTTW:2020} combined with the equivelant discs approximation -- that is, steps (1.) and (2.) in Algorithm 1 -- provides an extremely good initial guess of the curves to be recovered. This is essential for the convergence of Newton's method in view of the high nonlinearity of the shape identification problem.

\medskip

\paragraph{Using the third harmonic $M=3$:}

\begin{figure}[htbp]
\begin{center}
\includegraphics[width=0.19\textwidth]{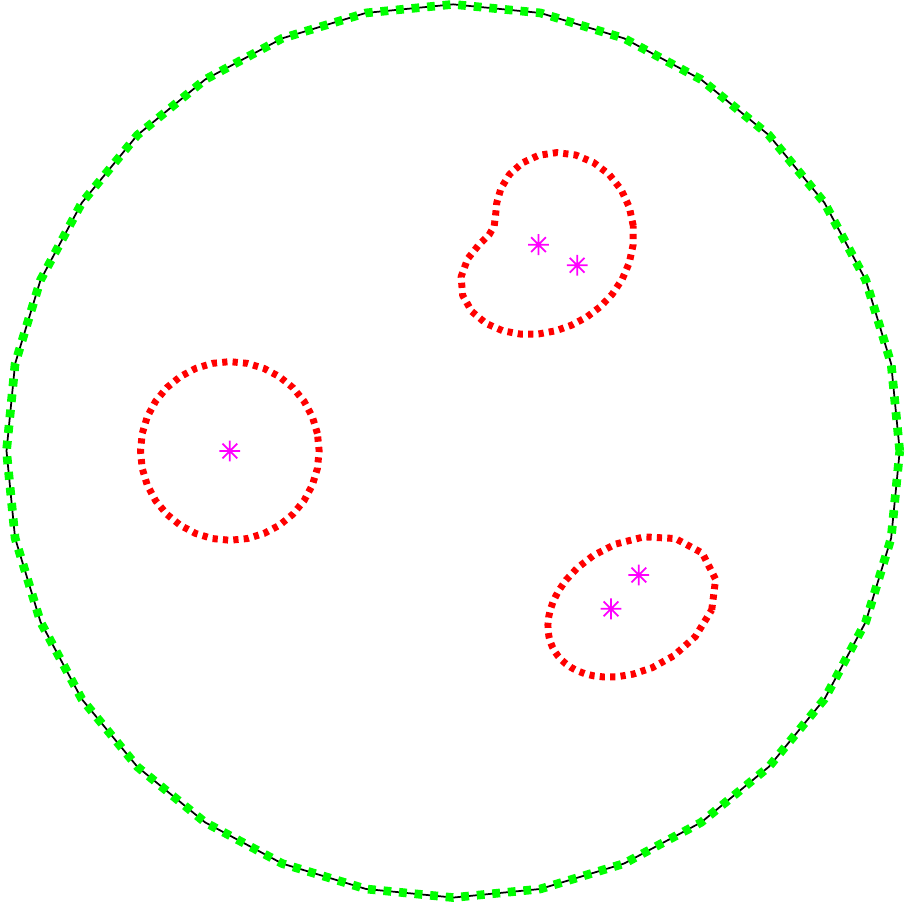}
\includegraphics[width=0.19\textwidth]{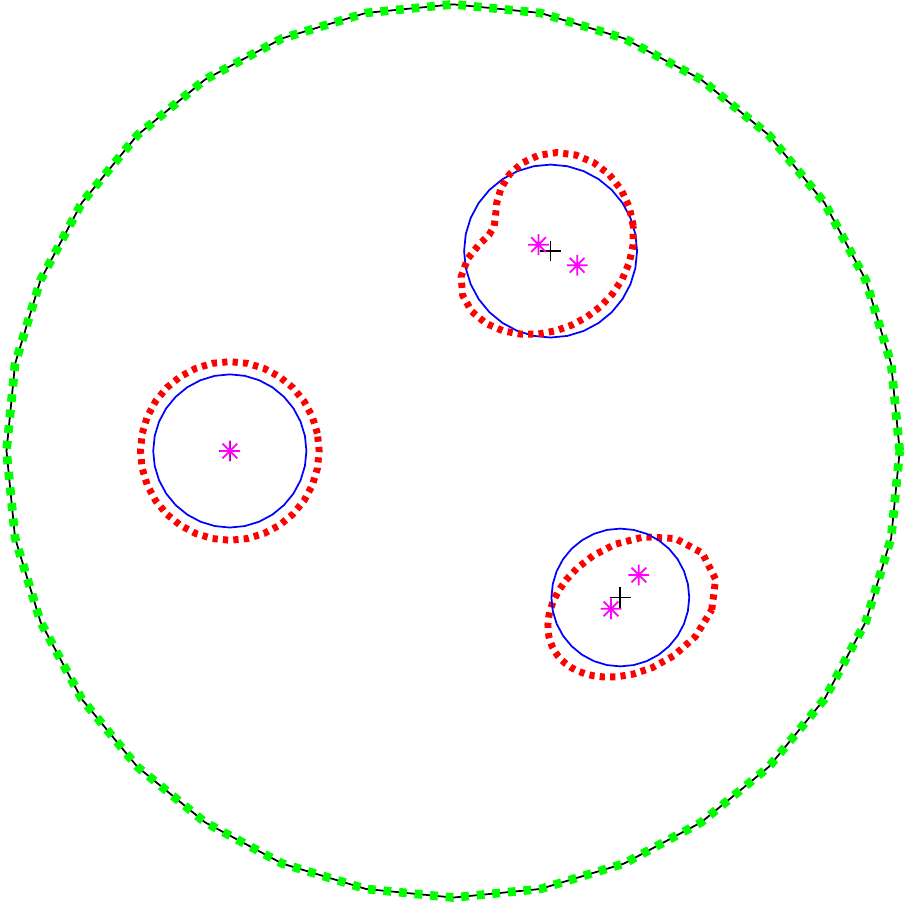}
\includegraphics[width=0.19\textwidth]{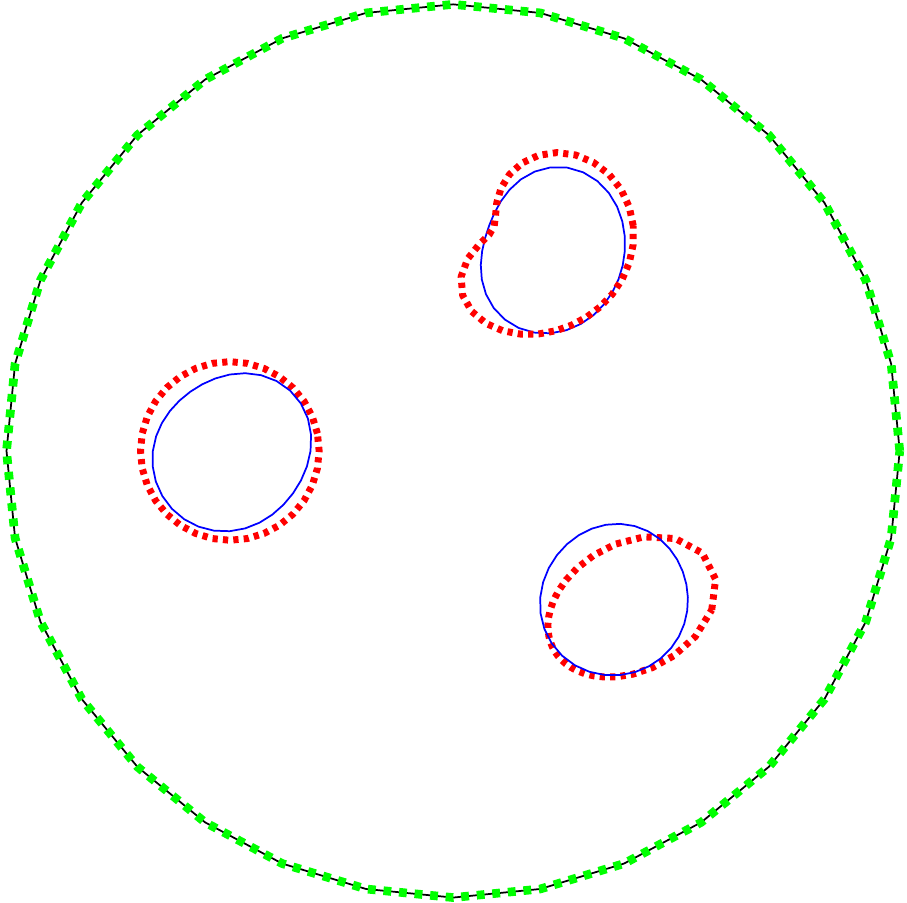}
\includegraphics[width=0.19\textwidth]{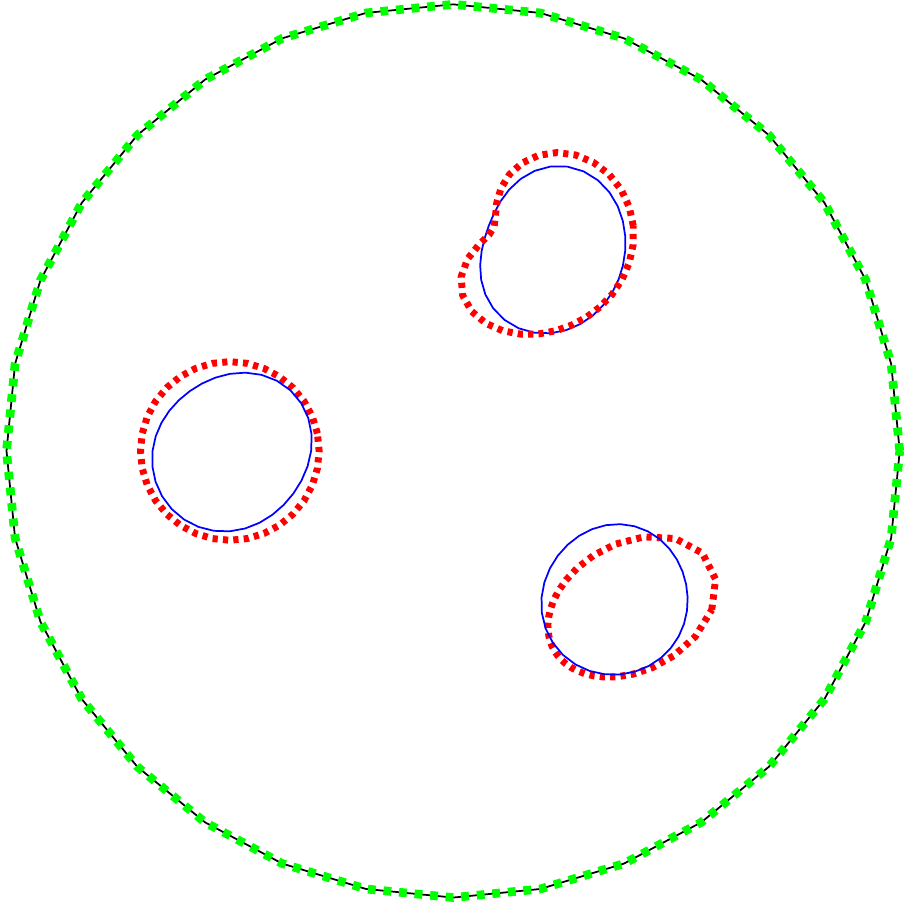}
\includegraphics[width=0.19\textwidth]{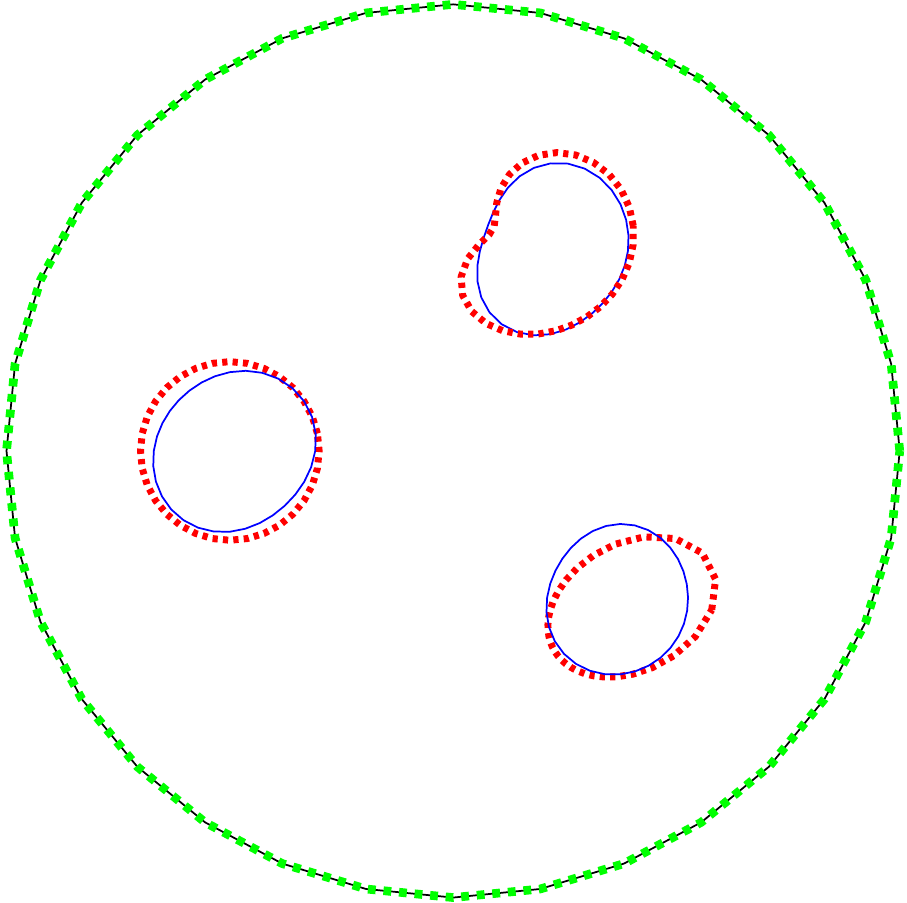}
\\[0.5ex]
\includegraphics[width=0.19\textwidth]{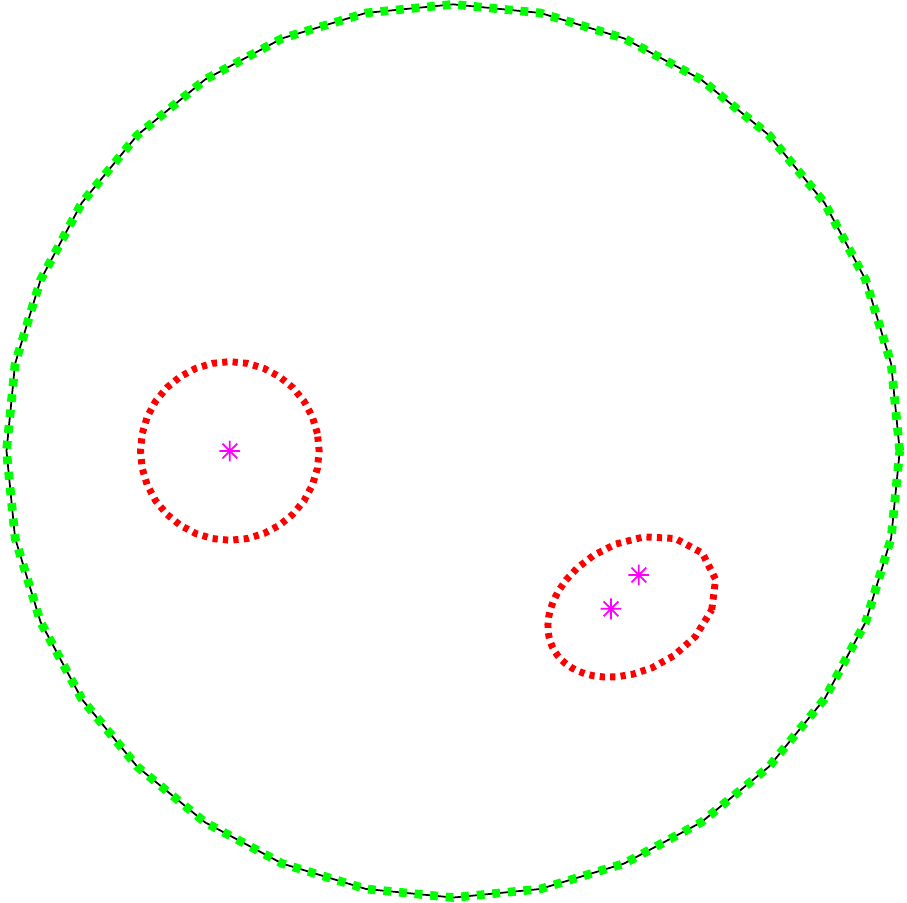}
\includegraphics[width=0.19\textwidth]{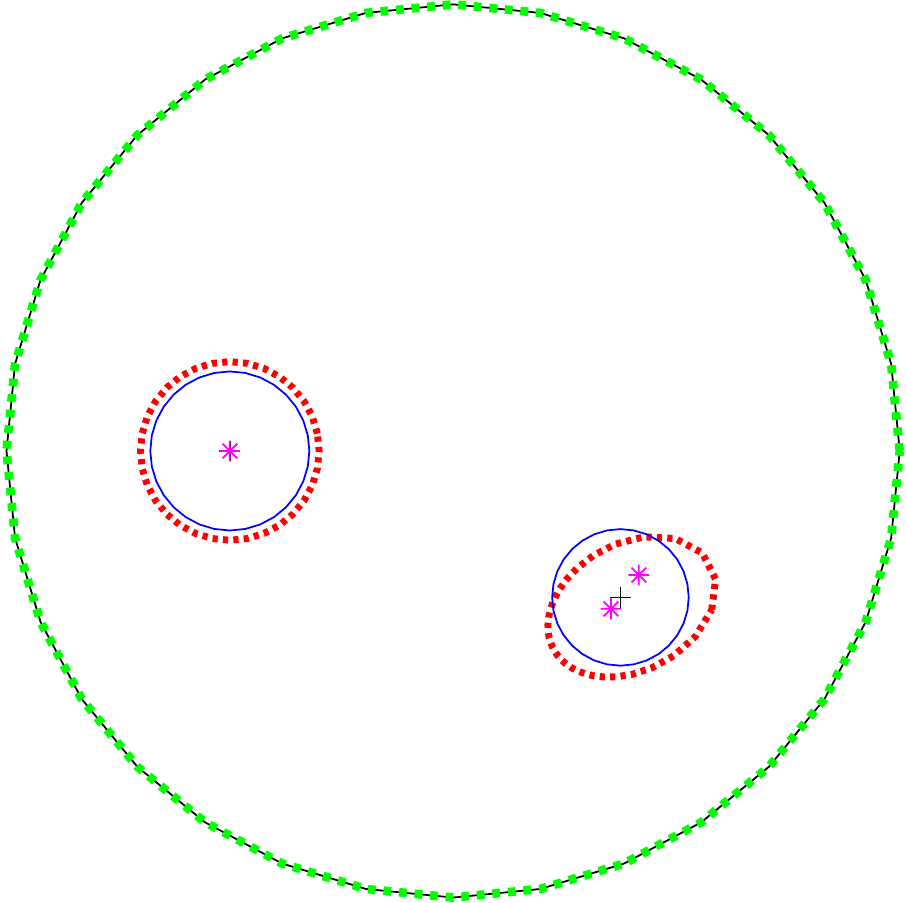}
\includegraphics[width=0.19\textwidth]{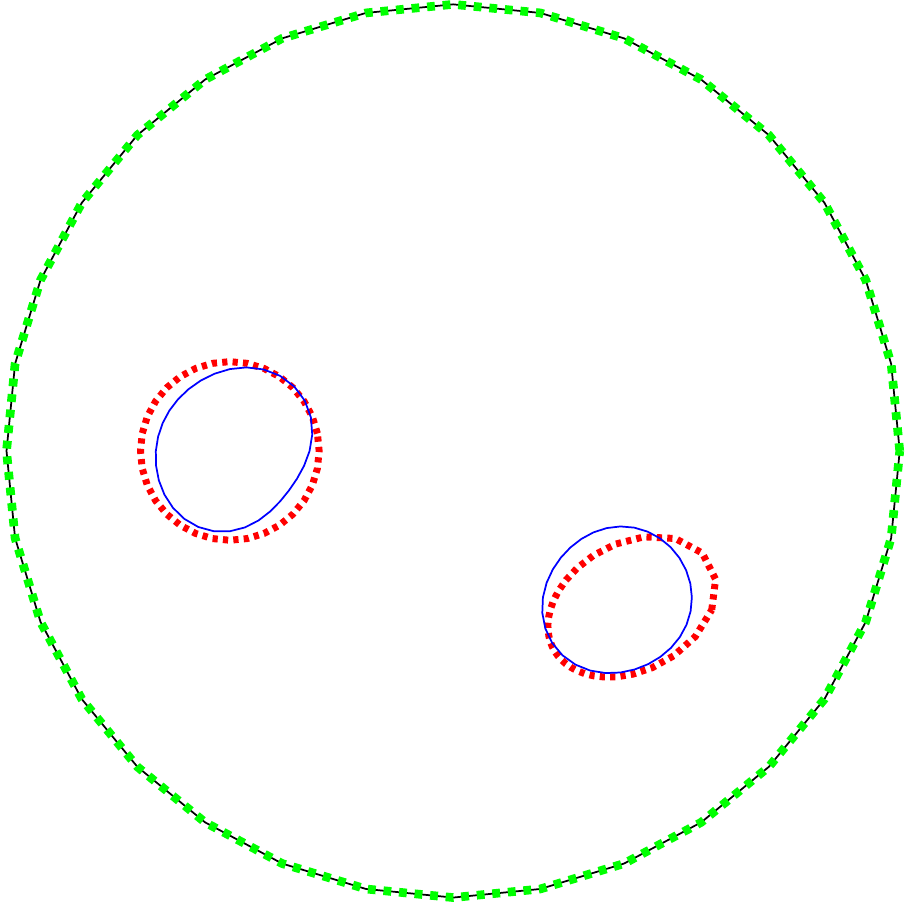}
\includegraphics[width=0.19\textwidth]{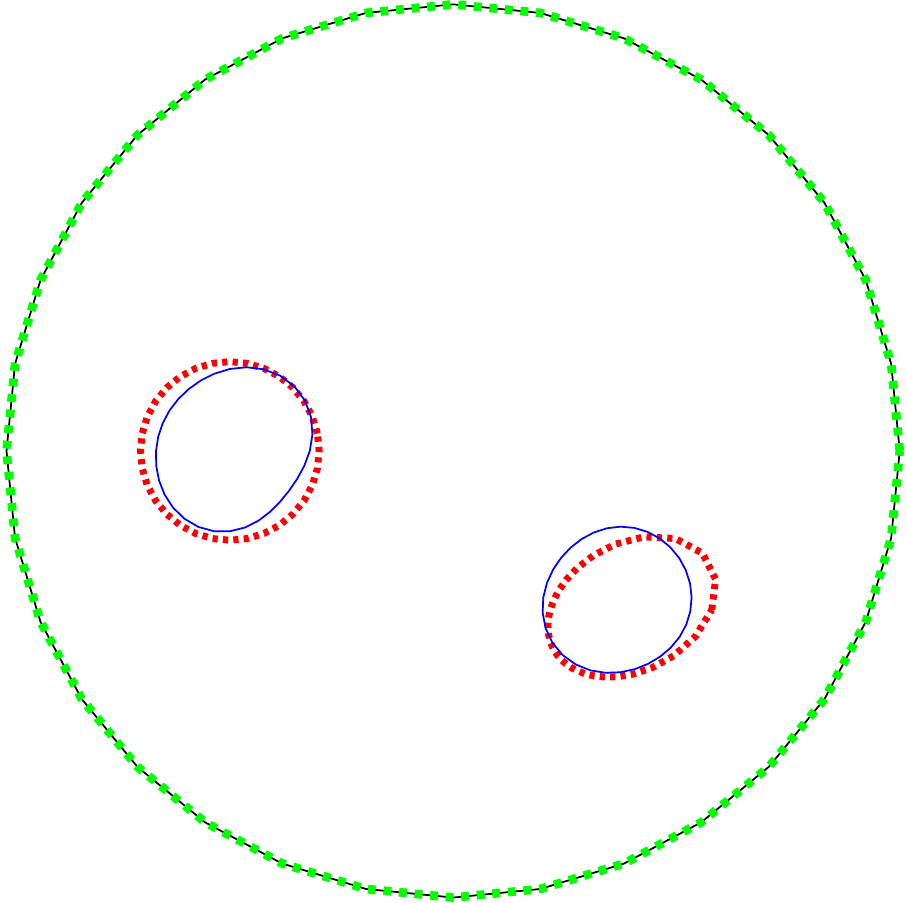}
\includegraphics[width=0.19\textwidth]{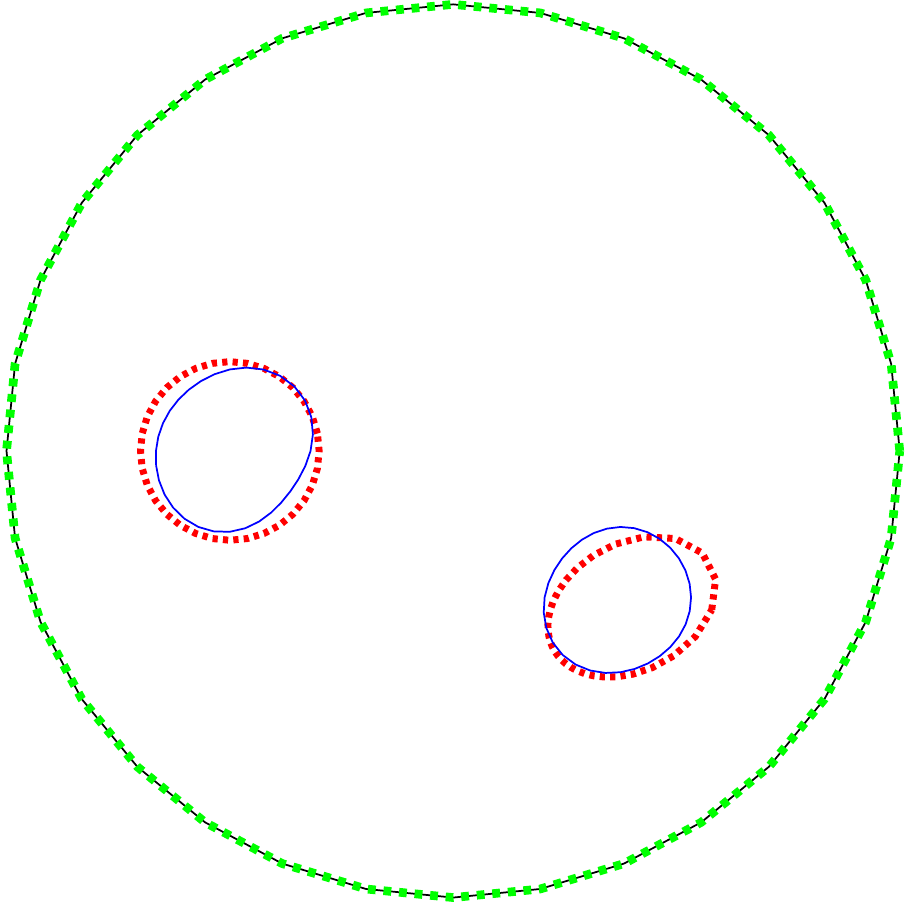}
\\[0.5ex]
\hspace*{0.08\textwidth}(a)\hspace*{0.165\textwidth}(b)\hspace*{0.165\textwidth}(c)\hspace*{0.165\textwidth}(d)\hspace*{0.165\textwidth}(e)\hspace*{0.08\textwidth}
\caption{Reconstruction of three (top row) or two (bottom row) inclusions from full data: (a) point sources step (1.) of Algorithm 1; (b) equivalent disks step (2.) of Algorithm 1; (c) Newton with second harmonic; (d) Newton with third harmonic; (e) Newton with second and third harmonic
\label{fig:3obj}}
\end{center}
\end{figure}

The reconstructions in Figure~\ref{fig:3obj} are obtained by following the steps of Algorithm 1 at wave number $\kaptil=10$ and then carrying out another Newton step with data from the third harmonic at $\kaptil=15$ either: (d) sequentially, using the result from  $\kaptil=10$ as a starting value or, (e) applying Newton's method simultaneously to $\kaptil=10$ and $\kaptil=15$.

The numerical results indicate that the additional information obtained from the next ($m=3$) harmonic does not yield much improvement. This is due to the lower -- by two to three orders of magnitude -- intensity of the signal at that higher frequency and seems to confirm the experimental evidence and common practice of skipping higher than second harmonics.  

\bigskip

\begin{figure}[htbp]
\begin{center}
\includegraphics[width=0.19\textwidth]{polpl_3obj_Tmeas1_Disks.eps}
\includegraphics[width=0.19\textwidth]{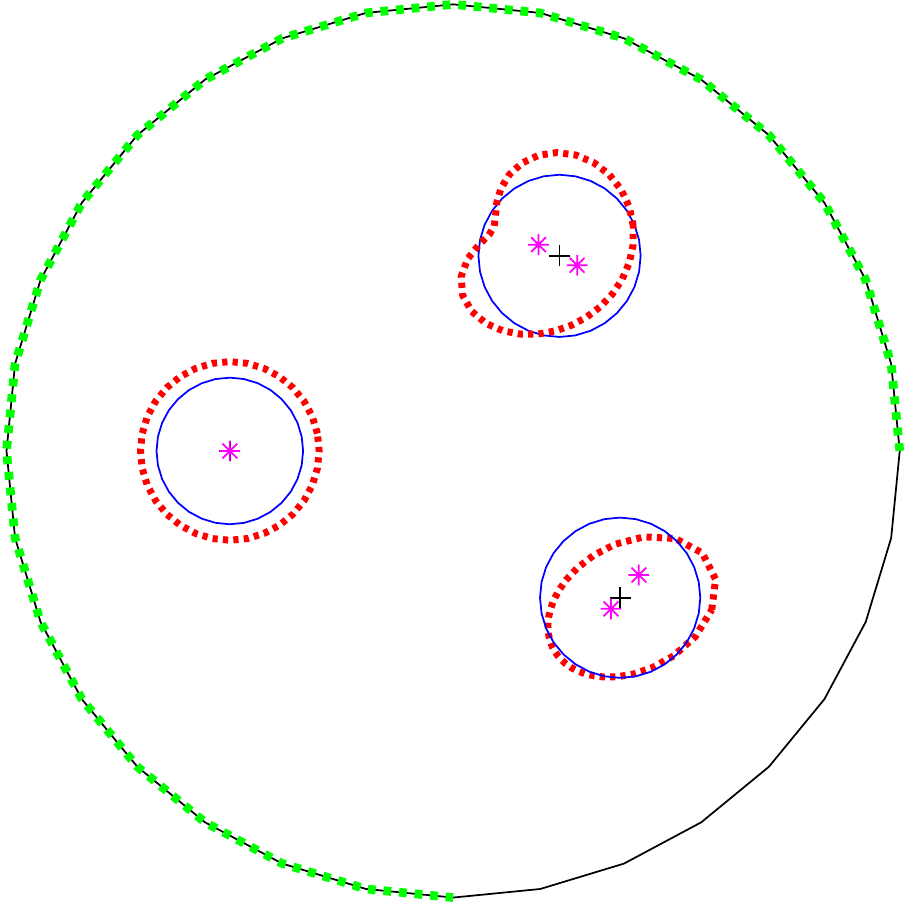}
\includegraphics[width=0.19\textwidth]{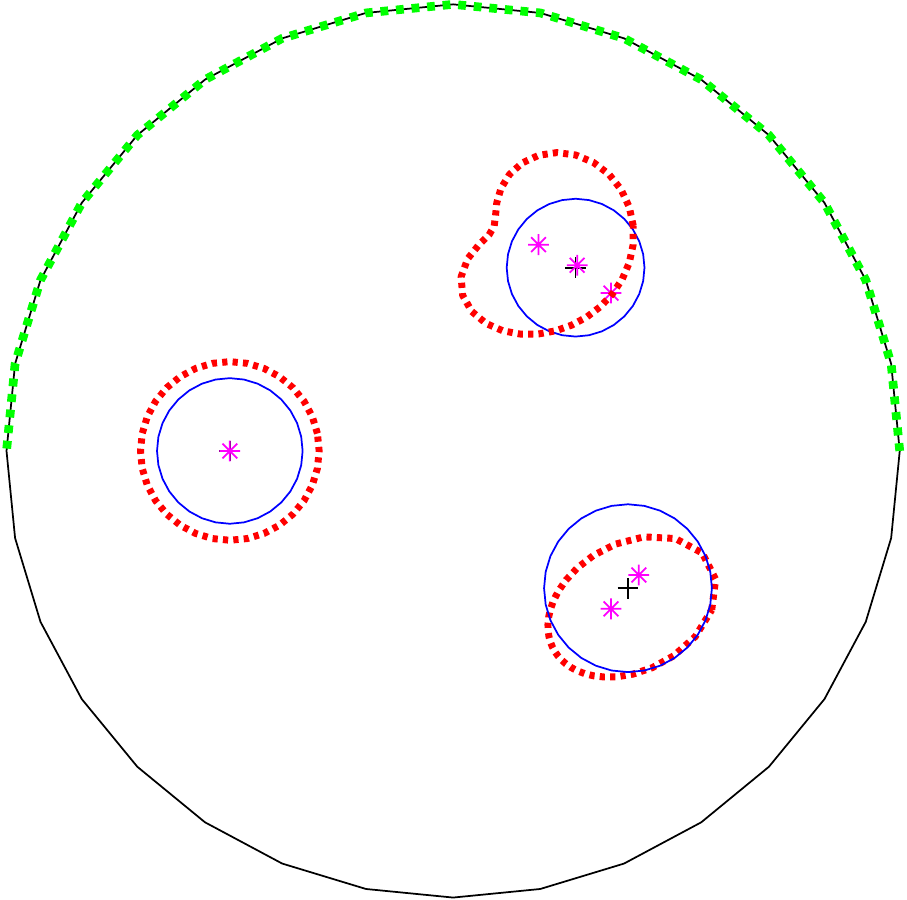}
\includegraphics[width=0.19\textwidth]{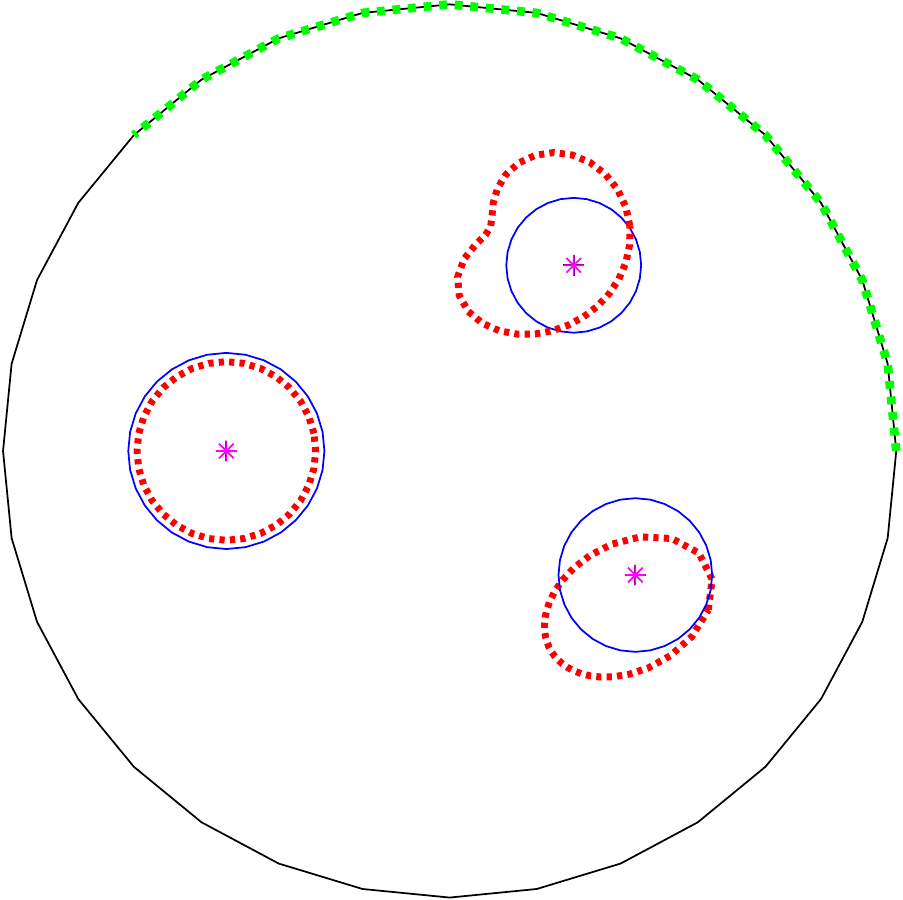}
\includegraphics[width=0.19\textwidth]{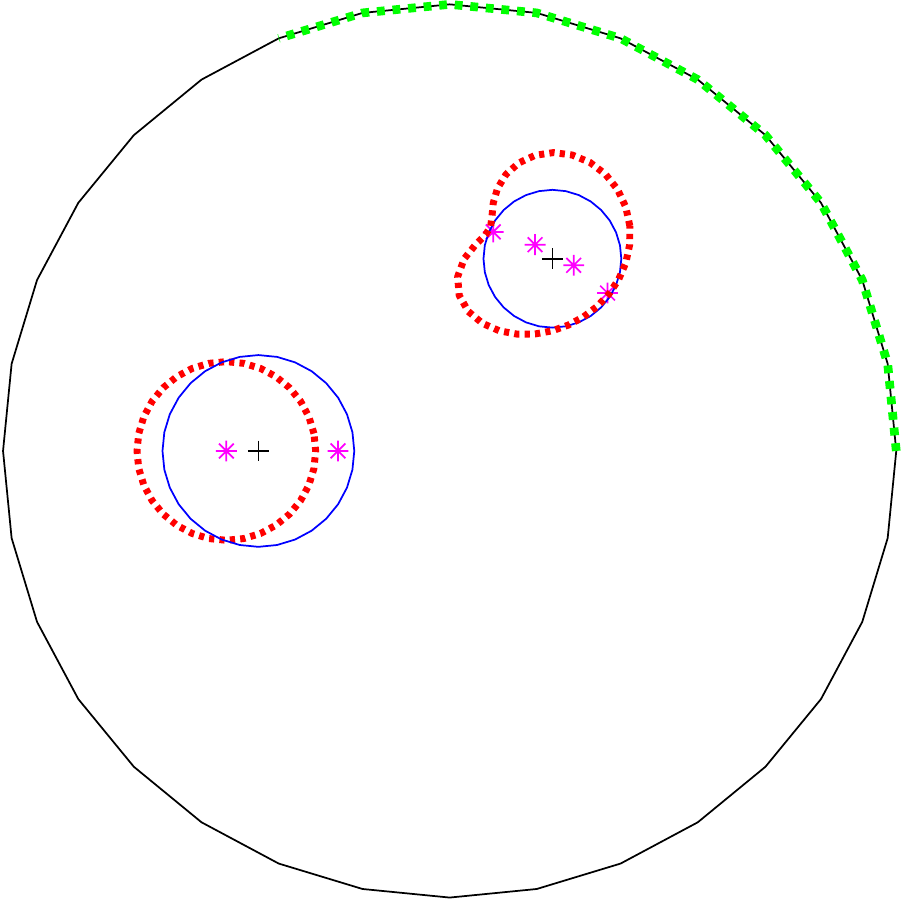}
\\[2ex]
\includegraphics[width=0.19\textwidth]{polpl_3obj_Tmeas1_Newton10.eps}
\includegraphics[width=0.19\textwidth]{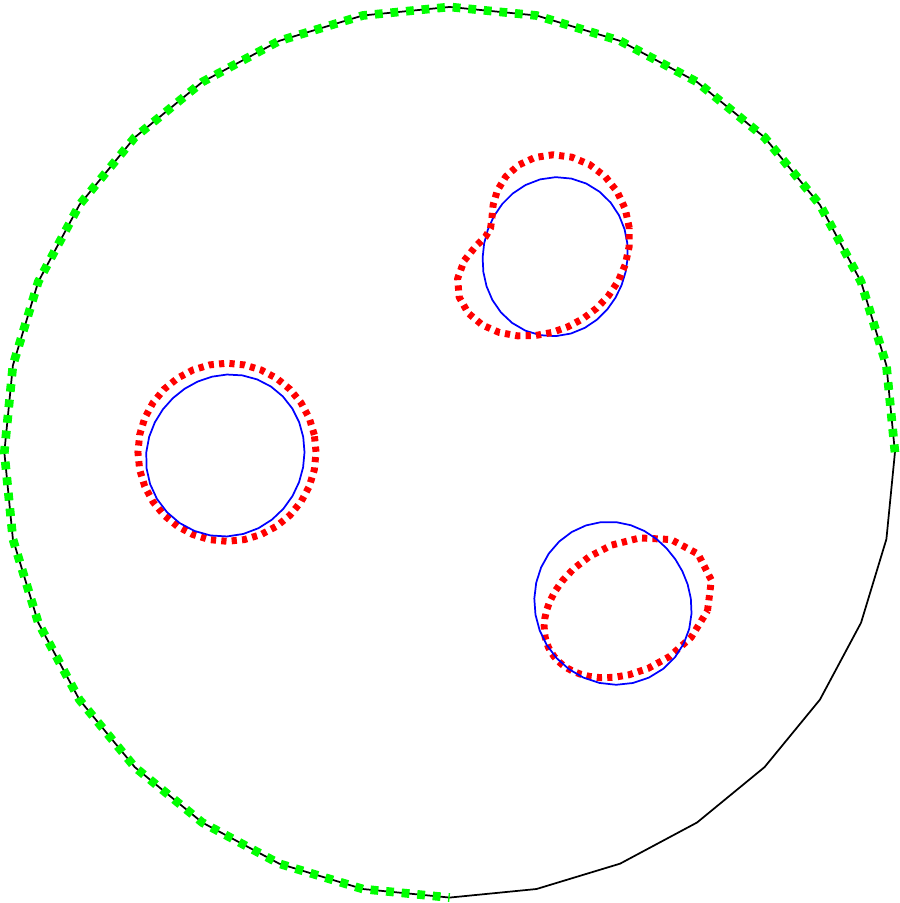}
\includegraphics[width=0.19\textwidth]{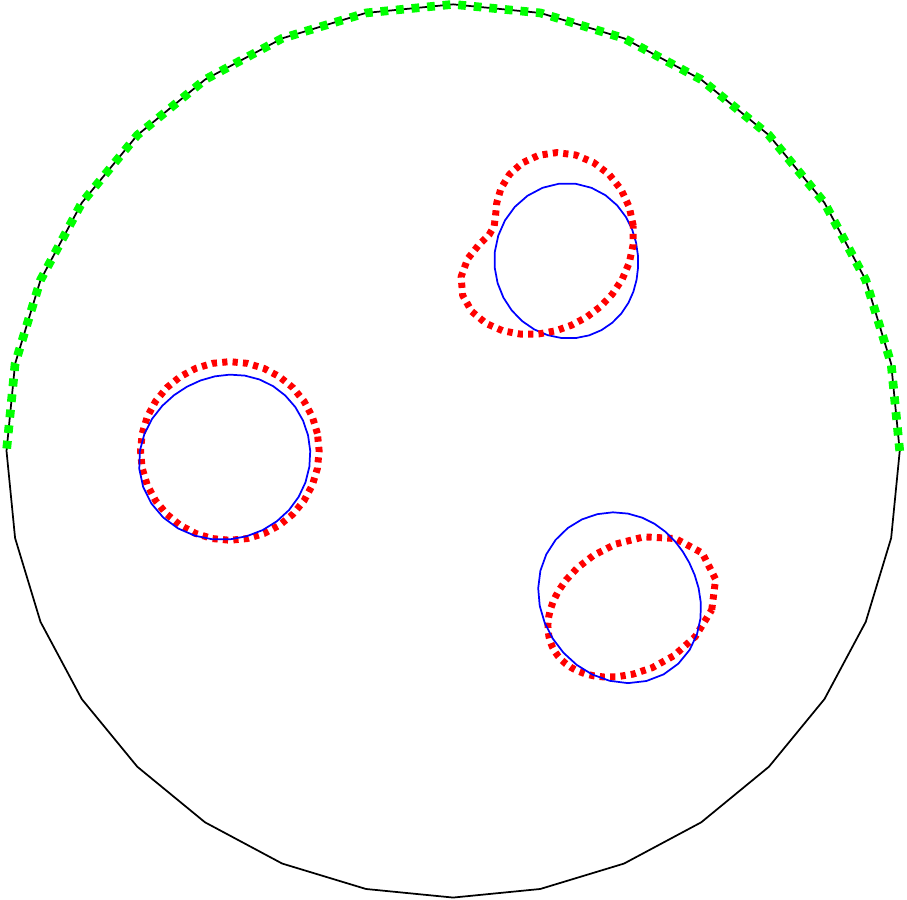}
\includegraphics[width=0.19\textwidth]{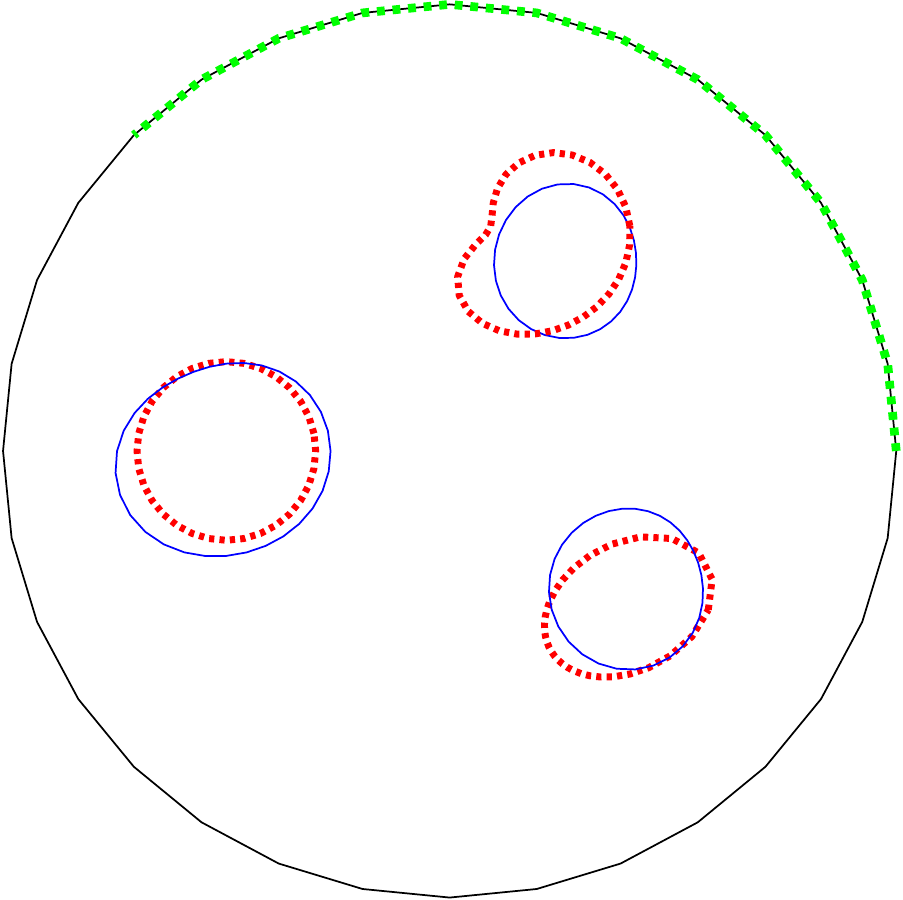}
\includegraphics[width=0.19\textwidth]{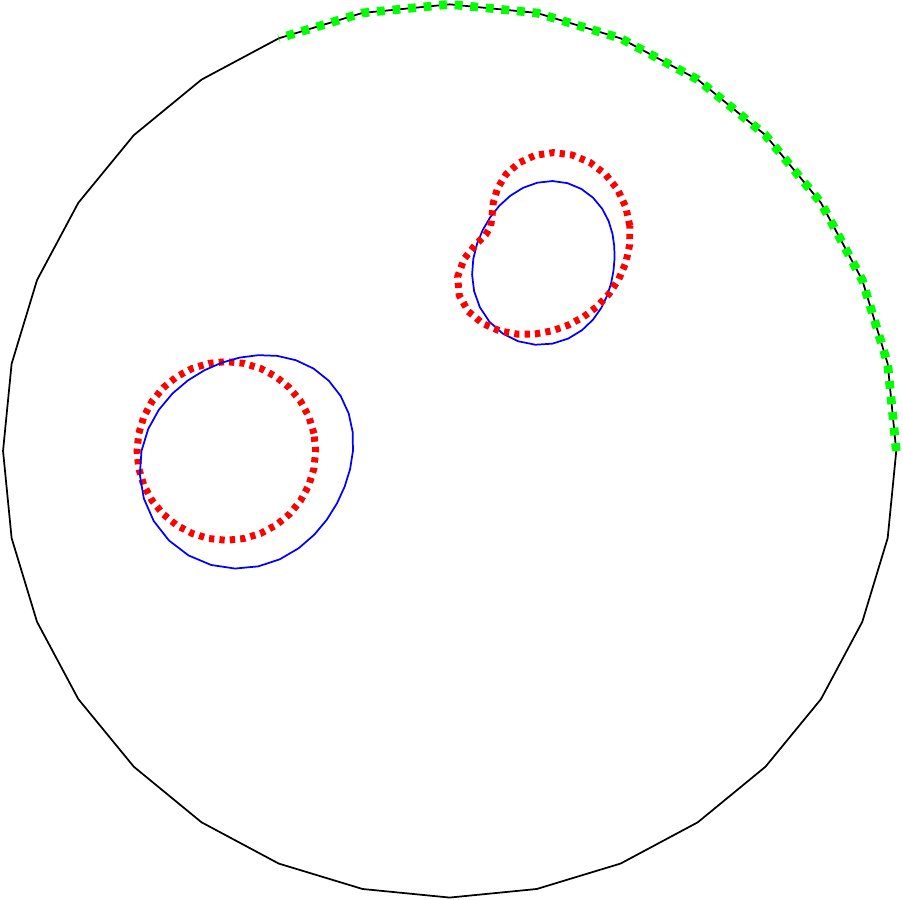}
\\[2ex]
(a) $\frac{\alpha}{2\pi}=1$ \hspace*{0.05\textwidth}
(b) $\frac{\alpha}{2\pi}=0.75$ \hspace*{0.05\textwidth}
(c) $\frac{\alpha}{2\pi}=0.5$ \hspace*{0.05\textwidth}
(d) $\frac{\alpha}{2\pi}=0.4$ \hspace*{0.05\textwidth}
(e) $\frac{\alpha}{2\pi}=0.3$ \hspace*{-0.01\textwidth}
\caption{Reconstruction of three inclusions from partial data; 
top row: equivalent point sources and disks; bottom row: boundary curves from Newton's method
\label{fig:3obj_partial}}
\end{center}
\end{figure}

\begin{figure}[htbp]
\begin{center}
\includegraphics[width=0.19\textwidth]{polpl_2obj_Tmeas1_Disks.eps}
\includegraphics[width=0.19\textwidth]{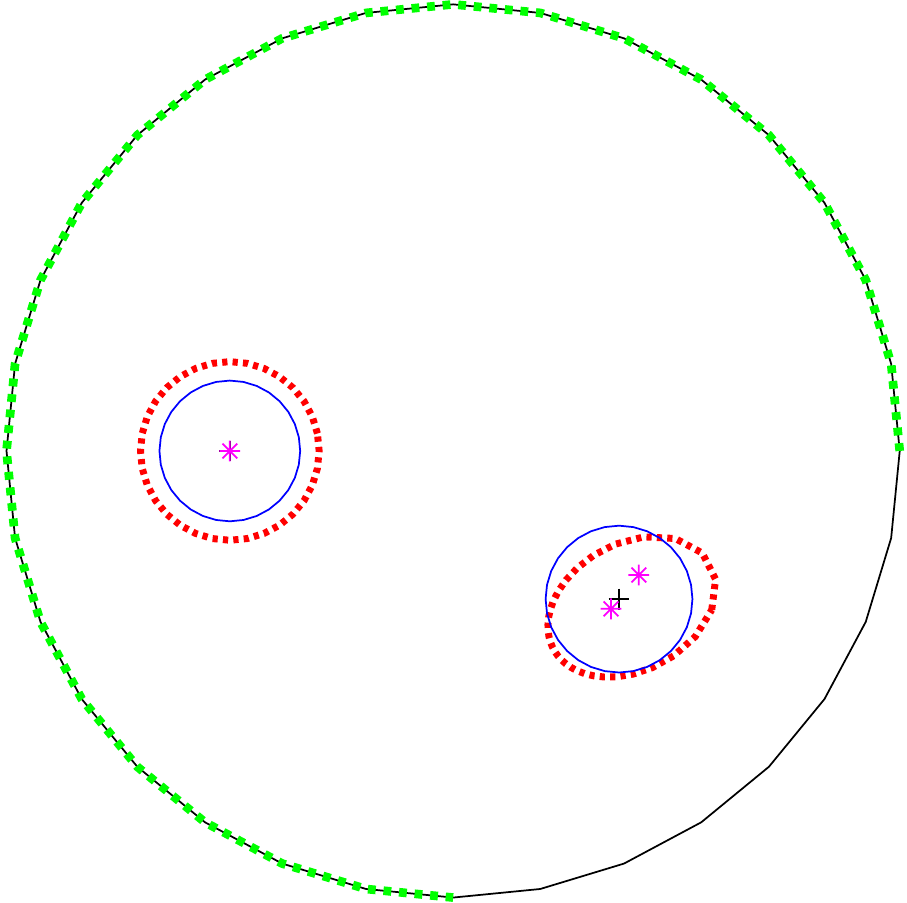}
\includegraphics[width=0.19\textwidth]{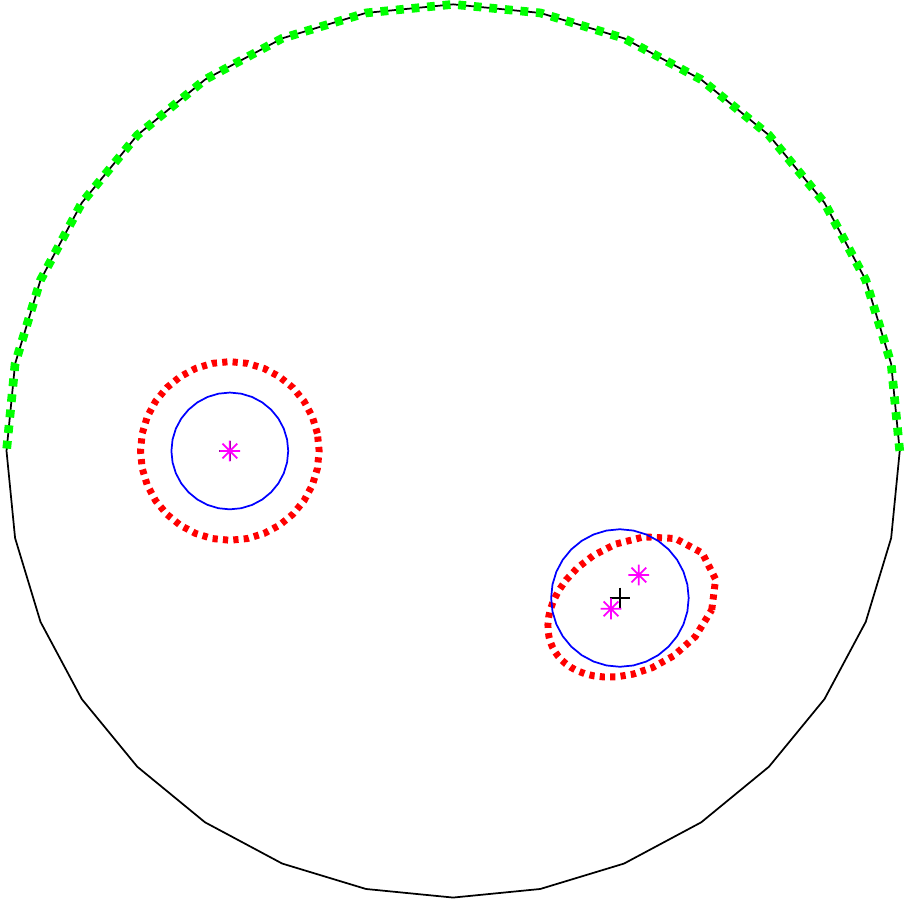}
\includegraphics[width=0.19\textwidth]{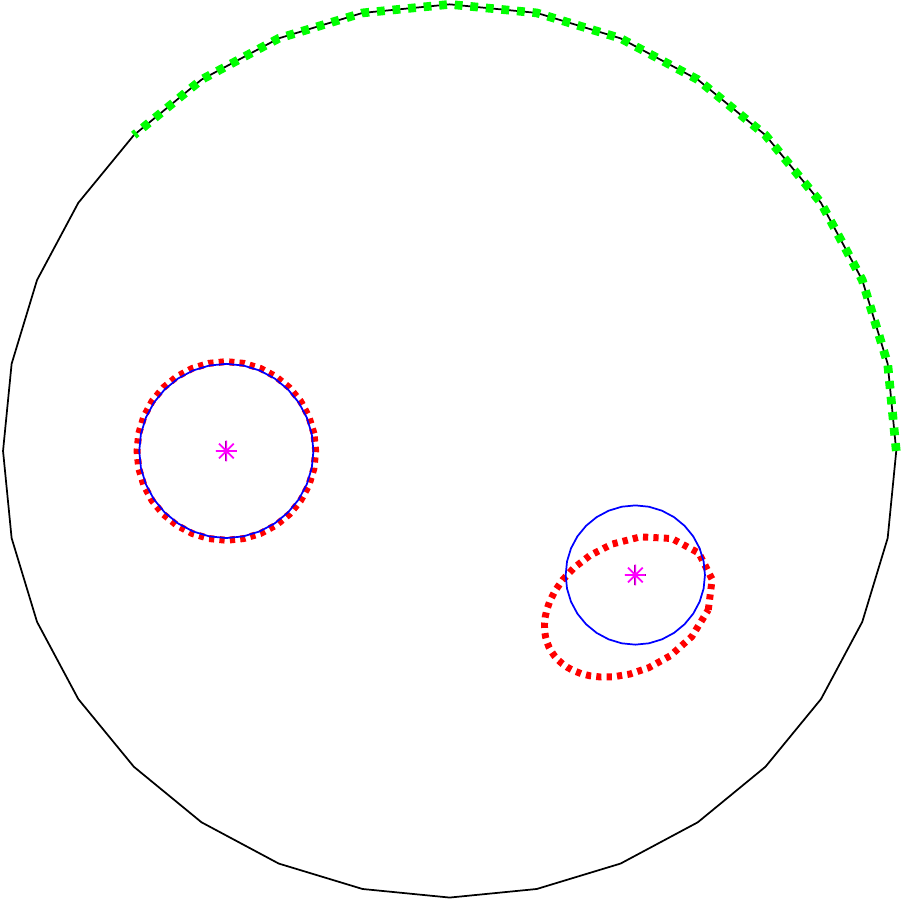}
\includegraphics[width=0.19\textwidth]{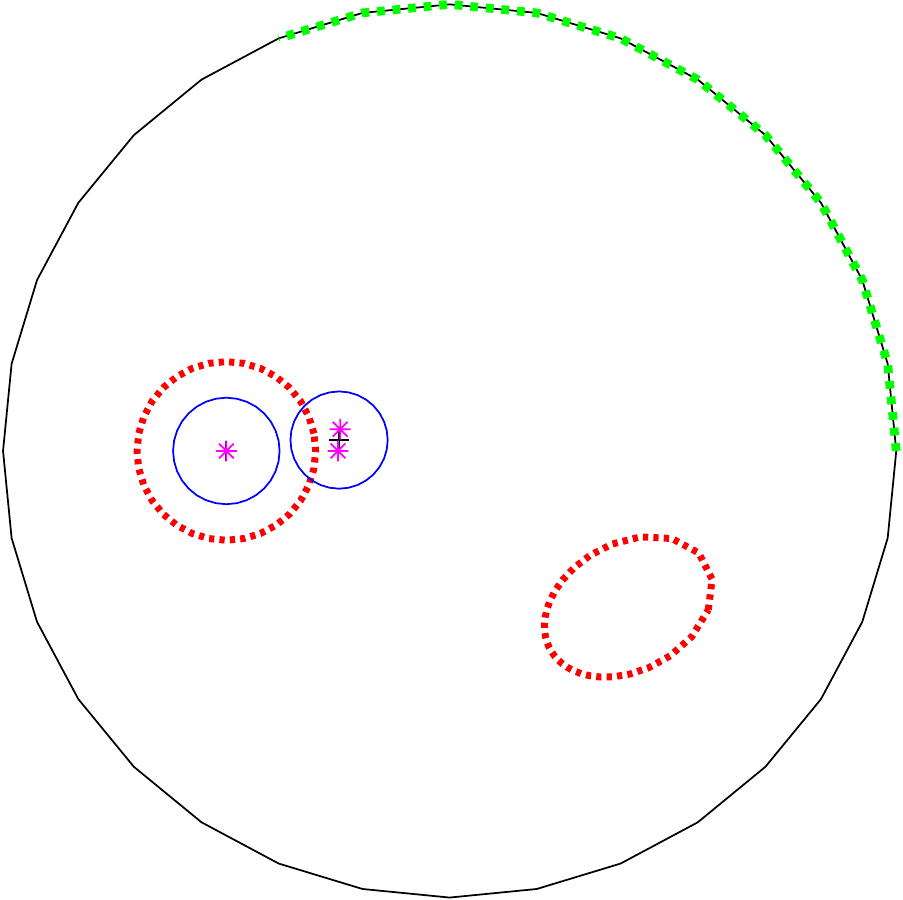}
\\[2ex]
\includegraphics[width=0.19\textwidth]{polpl_2obj_Tmeas1_Newton10.eps}
\includegraphics[width=0.19\textwidth]{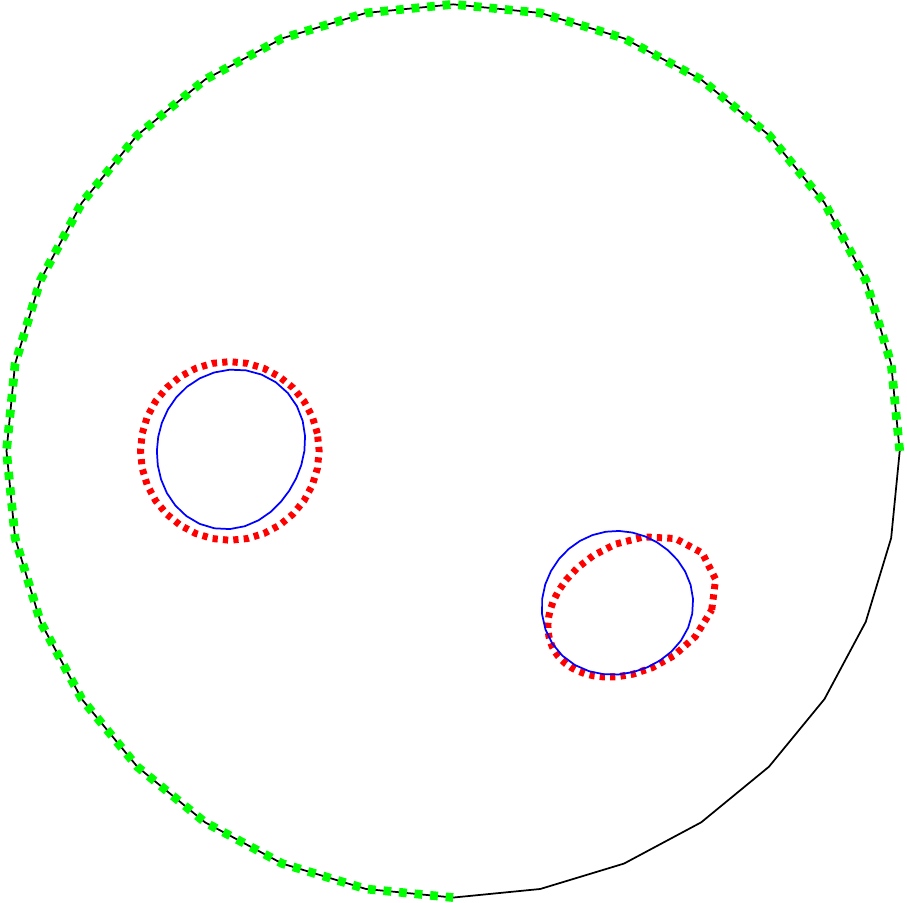}
\includegraphics[width=0.19\textwidth]{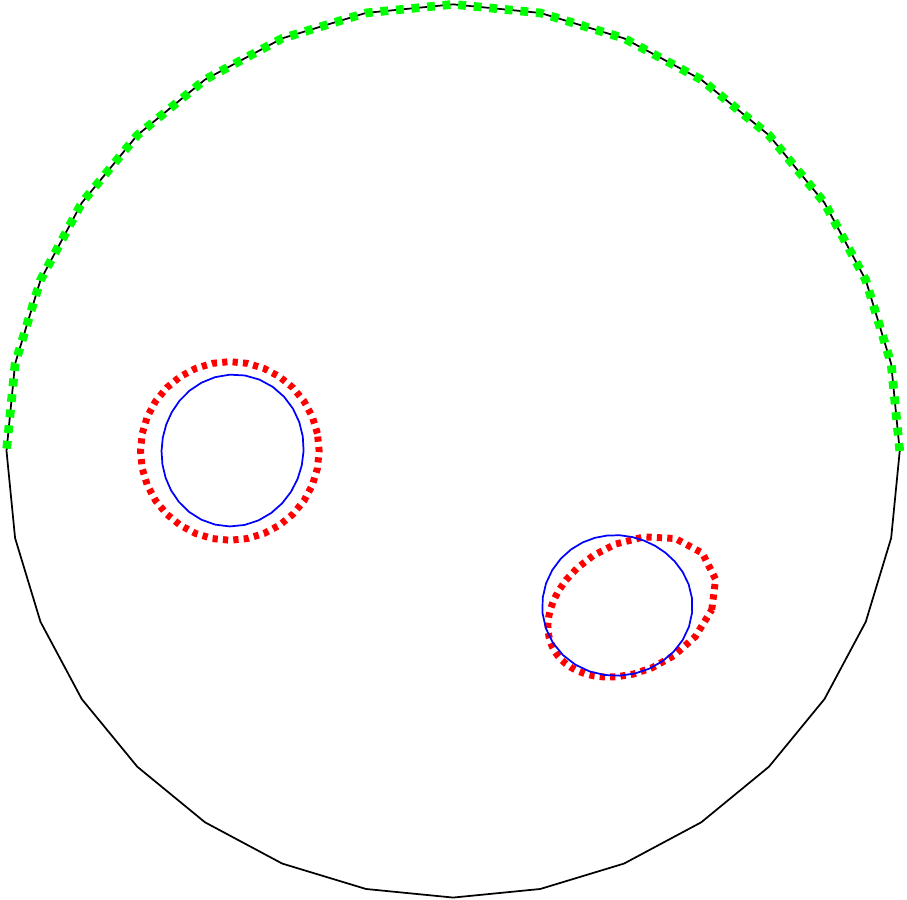}
\includegraphics[width=0.19\textwidth]{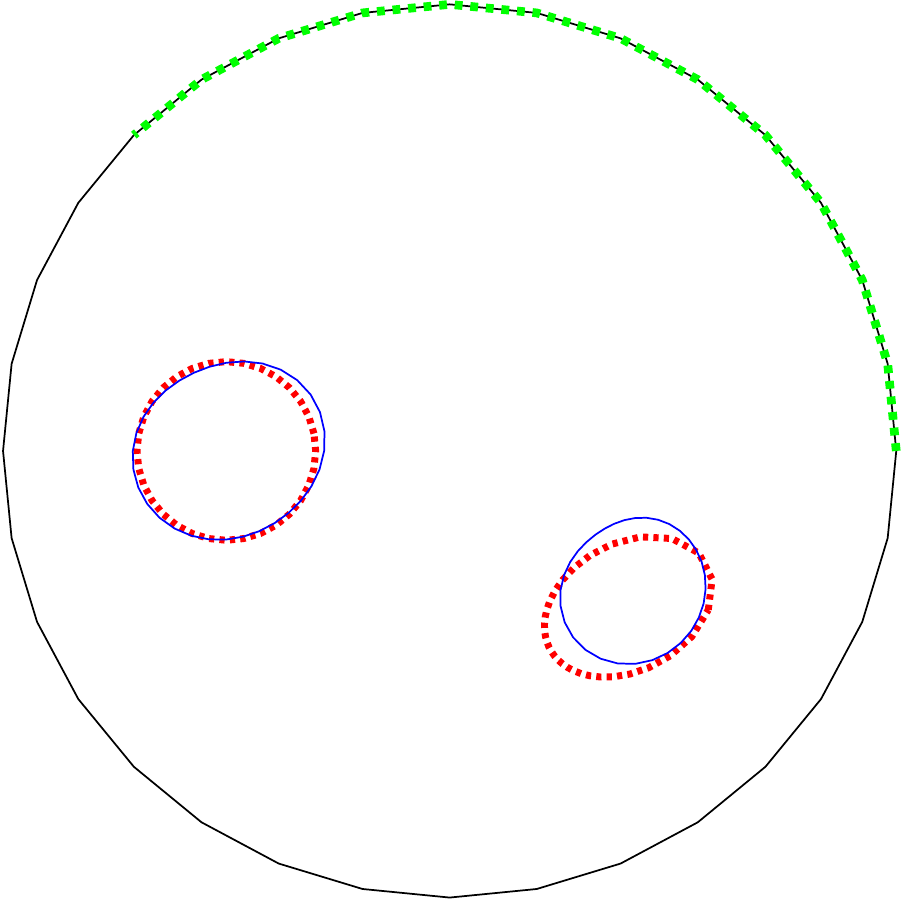}
\includegraphics[width=0.19\textwidth]{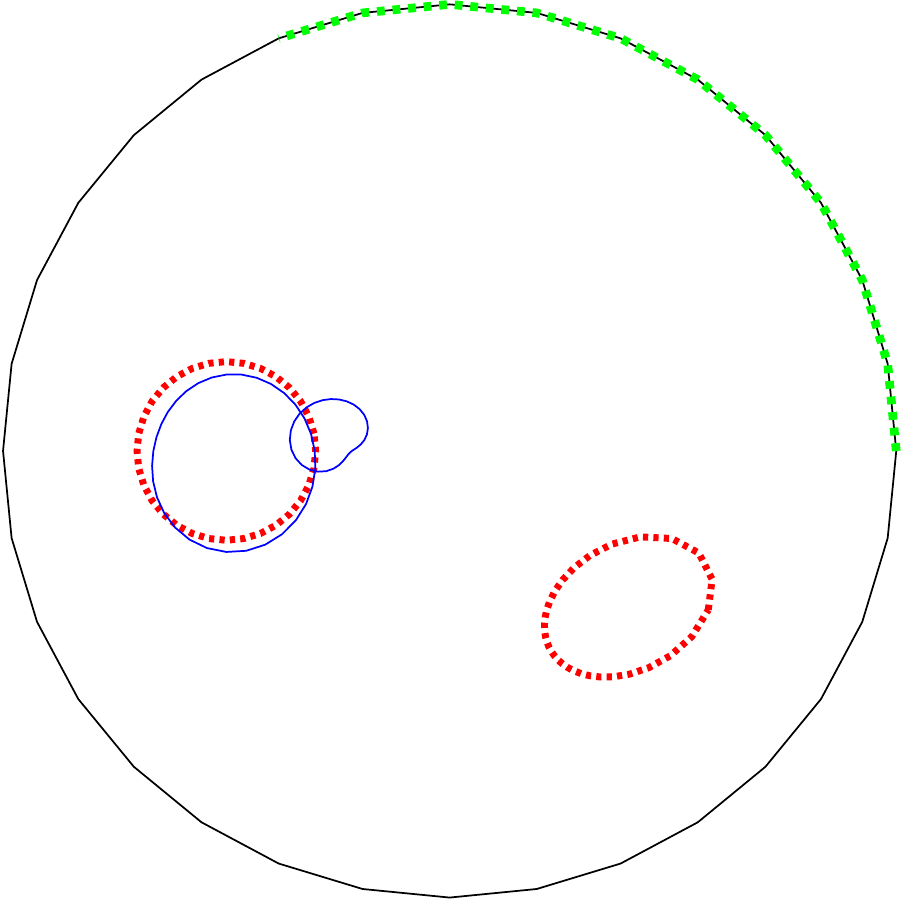}
\\[2ex]
(a) $\frac{\alpha}{2\pi}=1$ \hspace*{0.05\textwidth}
(b) $\frac{\alpha}{2\pi}=0.75$ \hspace*{0.05\textwidth}
(c) $\frac{\alpha}{2\pi}=0.5$ \hspace*{0.05\textwidth}
(d) $\frac{\alpha}{2\pi}=0.4$ \hspace*{0.05\textwidth}
(e) $\frac{\alpha}{2\pi}=0.3$ \hspace*{-0.01\textwidth}
\caption{Reconstruction of two inclusions from partial data;
top row: equivalent point sources and disks; bottom row: boundary curves from Newton's method
\label{fig:2obj_partial}}
\end{center}
\end{figure}

\paragraph{Reconstructions from partial data:} 

In Figures~\ref{fig:3obj_partial}, \ref{fig:2obj_partial} we show reconstructions from partial data.
The quality appears to decrease only slightly with decreasing amount of data, until at a certain point (between 30 and 40 per cent of the full angle) the algorithm partially breaks down and fails to find one of the objects completely.
The ability of an inclusion to stay reconstructible from a low amount of data is related to its weight according to the associated weight $\lambda_k$ according to \eqref{meanvalue_Helmholtz} (using the object's average radius). 
In Figures~\ref{fig:3obj_partial} and \ref{fig:2obj_partial} these weights are: 0.0725 for the circle, 0.0692 for the cardiod and 0.0515 for the ellipse.
Also, the position relative to the measurement boundary clearly plays a role.

It may seem that simple completion
of data from the measurement subarc to the entire boundary should give
similar results by for example using the Fourier series expansion.
However, this analytic continuation step comes at a price.
If we have $N$ Fourier modes over an arc of length $\alpha$
then this analytic continuation results from solving a system with a matrix $P(N,\alpha)$
the conditioning of which can be computed analytically.
Of course the condition number will increase with both $N$ and decreasing
values of $\alpha$, $0<\alpha<2\pi$.
In fact this is a well-understood problem, see \cite{Slepian:1983} where it has been shown that the condition number of $P(N,\alpha)$ is asymptotic (for large $N$) to
\begin{equation}\label{eqn:condSlepian}
c_N\sim e^{\gamma(\alpha) N}\textup{ where } 
\gamma(\alpha) = \log\Bigl(\frac{\sqrt{2}+\sqrt{1+\cos\alpha}}{\sqrt{2}-\sqrt{1+\cos\alpha}}\Bigr).
\end{equation}
This has been used in several inverse problems, see, e.g., \cite{HR:1997,Louis:1986}.

However, in our situation the reconstructions are performing much better than
the above pessismistic estimate would suggest.
This is due to the fact that our reconstruction does not rely on extending the boundary data but rather on  directly applying our method to the restricted flux $g=\partial_\nu \hat{p}\vert_\Sigma$. The additional information that the PDE model provides clearly contributes to this inprovement, which is also reflected in the condition number of the Jacobian in Newton's method versus the theoretical prediction for data completion from \cite{Slepian:1983}. This can be seen in Table~\ref{tab:conds}. 
\begin{table}[htbp]
\begin{center}
\begin{tabular}{|c|c|c|}\hline
$\frac{\alpha}{2\pi}$ & cond(J) & $c_N$ \cite{Slepian:1983}\\ \hline
0.75 & 29.6 & 2.8e+2\\
0.5 & 64.9 & 2.3e+5\\
0.4& 73.7 & 1.8e+07\\
0.3& 1733.8& 2.6e+08\\ \hline
\end{tabular}
\end{center}
\caption{Condition numbers of Jacobian in Newton's method for a single inclusion using $9$ basis functions versus condition number formula \eqref{eqn:condSlepian} for data completion with $N=9$
\label{tab:conds}}
\end{table}




\bigskip

\begin{figure}[htbp]
\begin{center}
\includegraphics[width=0.19\textwidth]{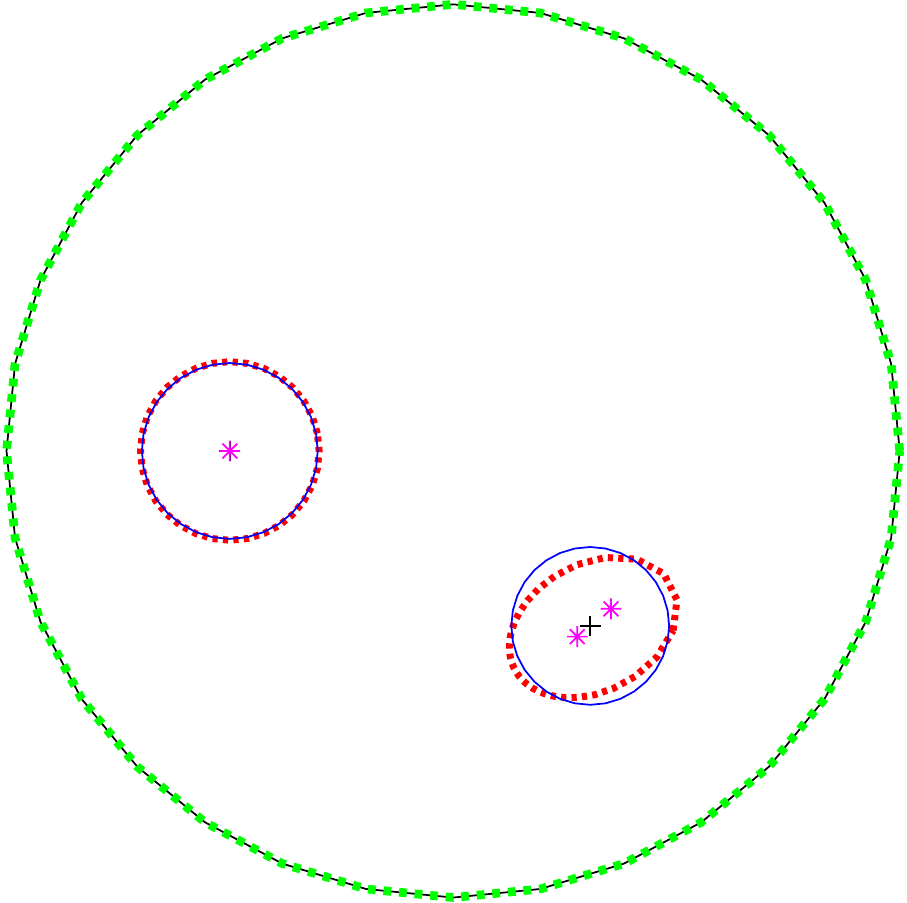}
\includegraphics[width=0.19\textwidth]{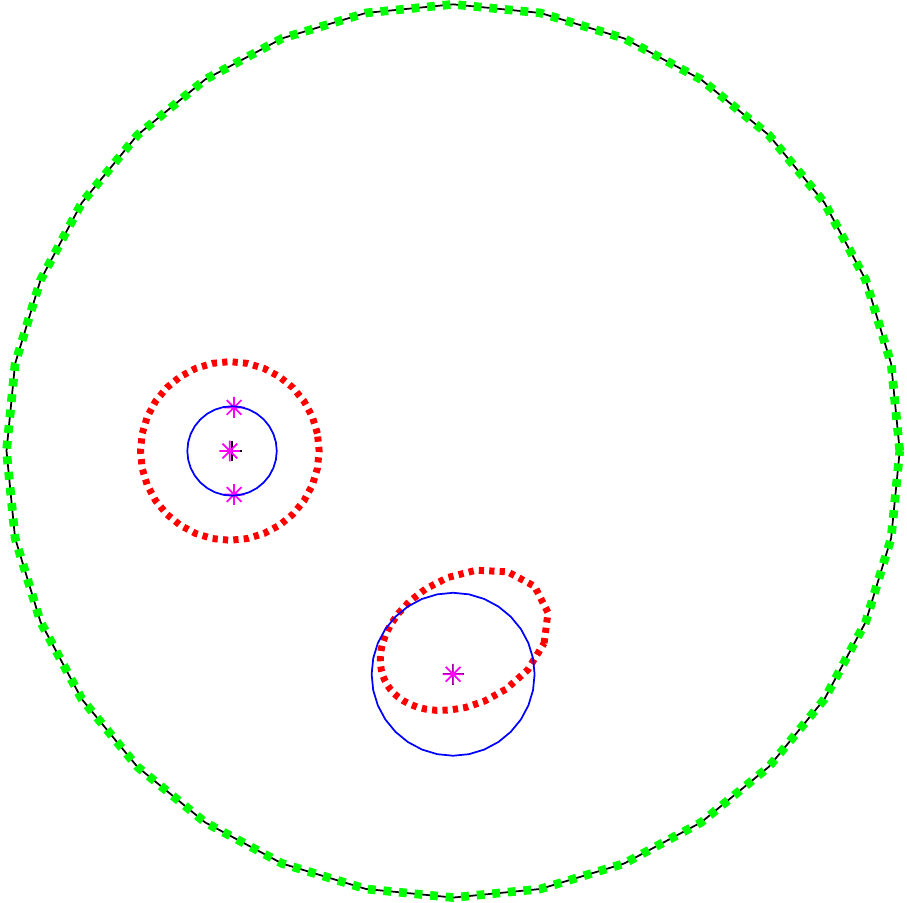}
\includegraphics[width=0.19\textwidth]{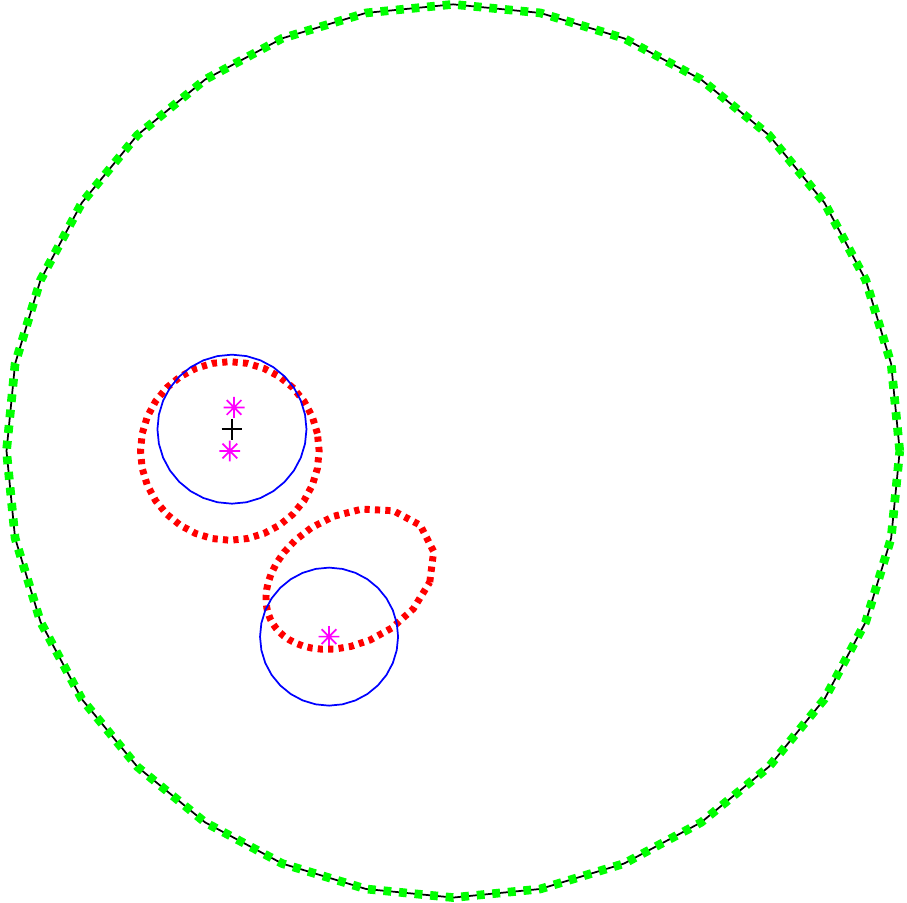}
\includegraphics[width=0.19\textwidth]{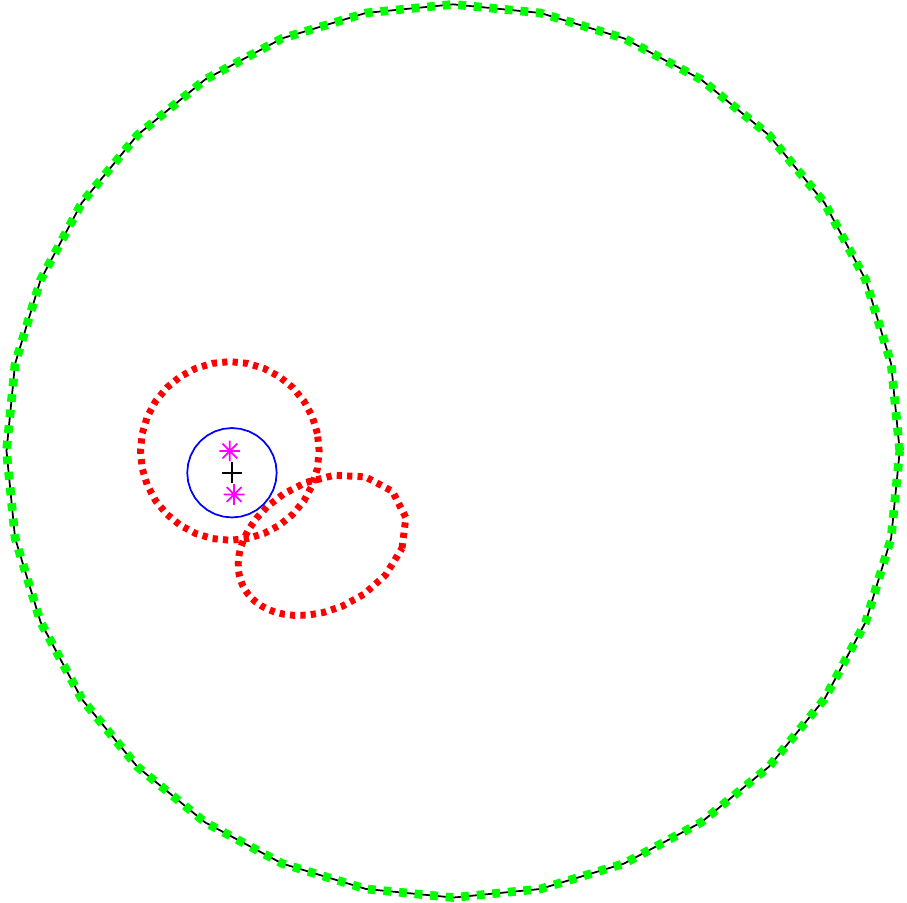}
\\[2ex]
\includegraphics[width=0.19\textwidth]{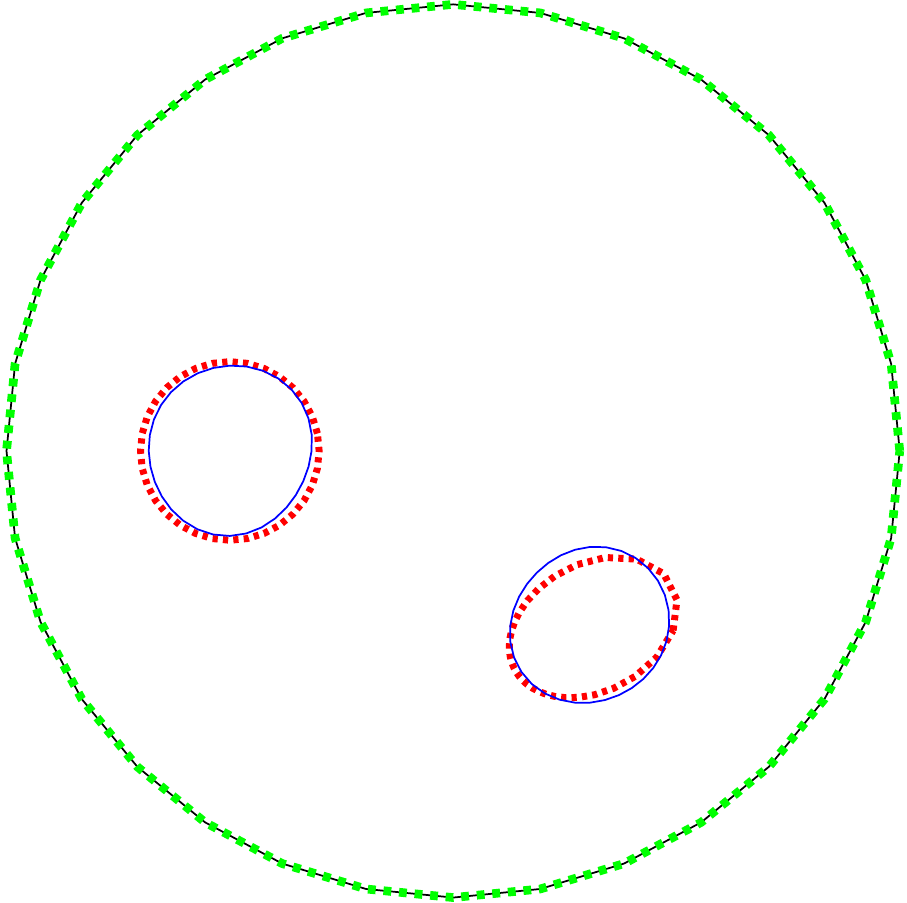}
\includegraphics[width=0.19\textwidth]{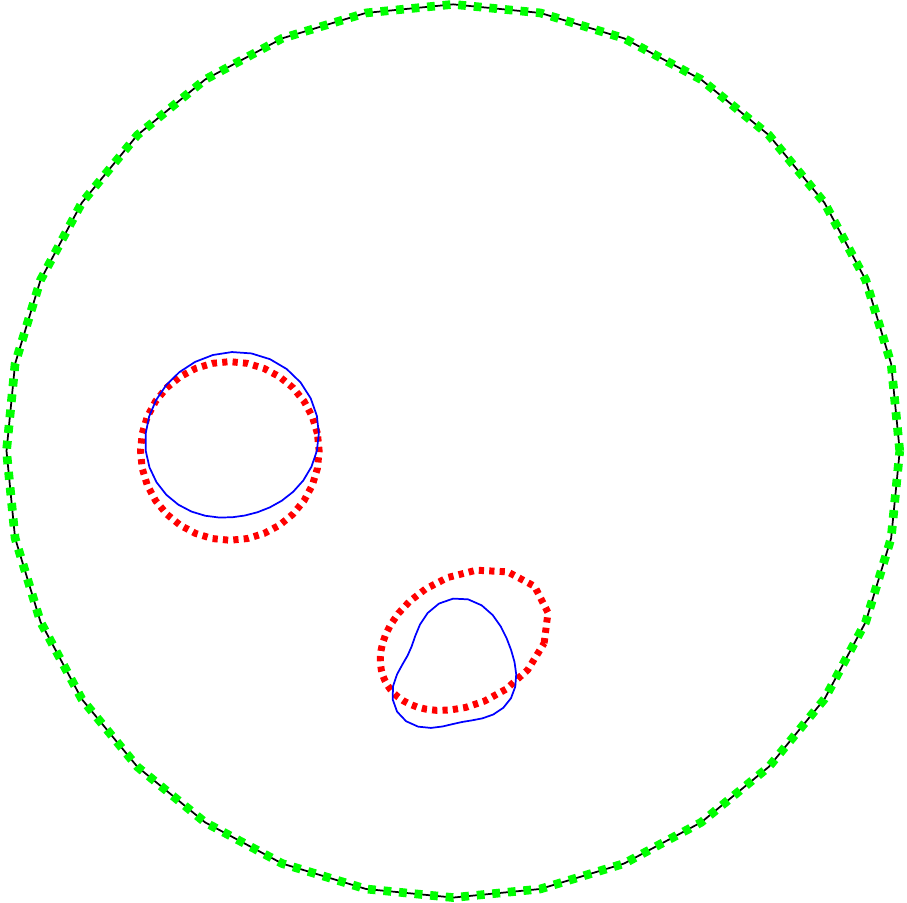}
\includegraphics[width=0.19\textwidth]{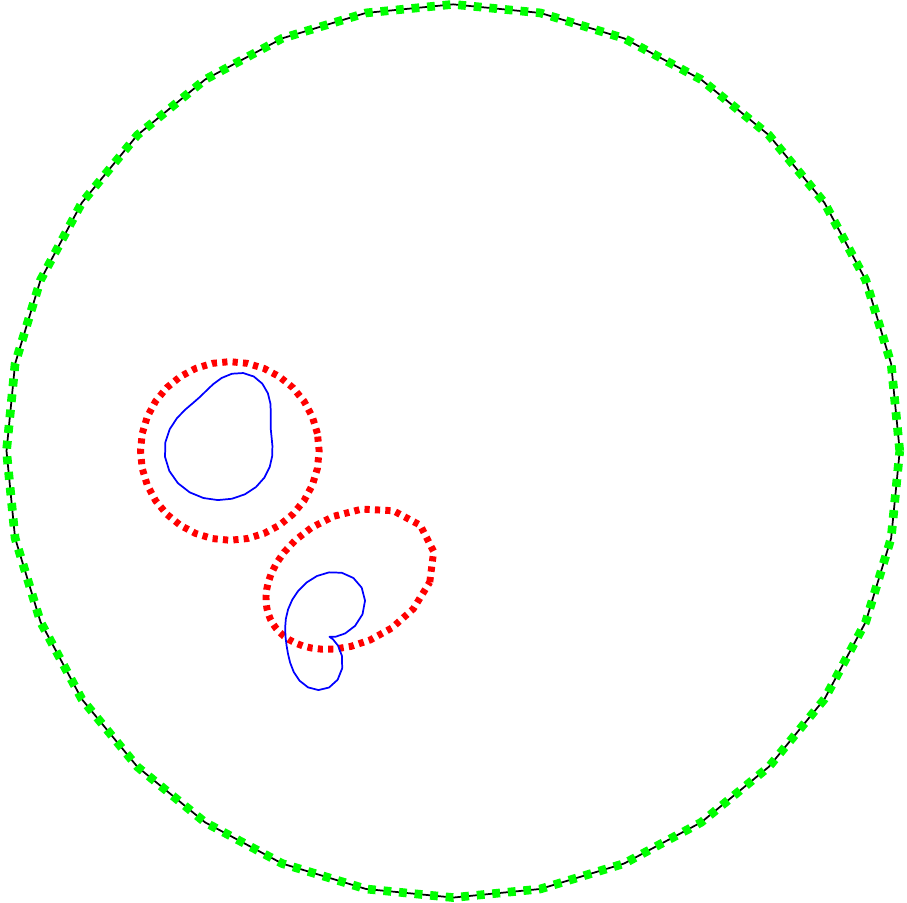}
\includegraphics[width=0.19\textwidth]{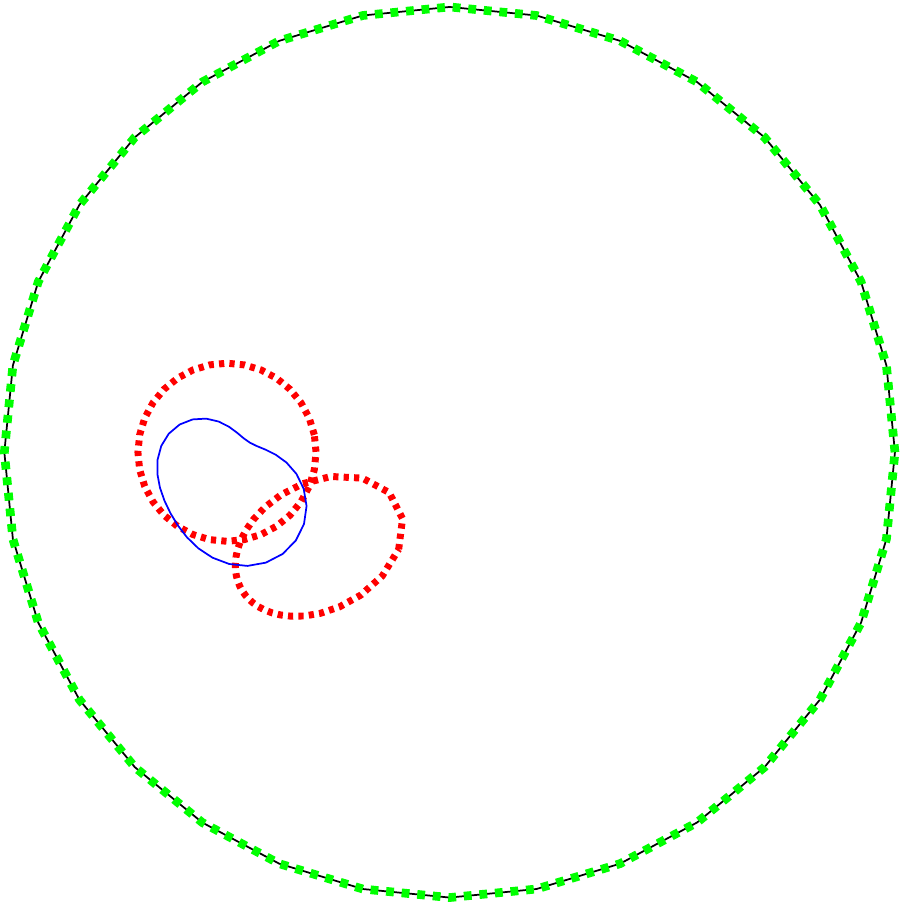}
\\[2ex]
(a) $\frac{\theta}{2\pi}=0.3$ \hspace*{0.05\textwidth}
(b) $\frac{\theta}{2\pi}=0.2$ \hspace*{0.05\textwidth}
(c) $\frac{\theta}{2\pi}=0.1$ \hspace*{0.05\textwidth}
(d) $\frac{\theta}{2\pi}=0.09$ \hspace*{-0.01\textwidth}
\caption{Reconstruction of two inclusions at different distances;
top row: equivalent point sources and disks; bottom row: boundary curves from Newton's method
\label{fig:2obj_dist}}
\end{center}
\end{figure}

\paragraph{Varying distance between objects:} 
Figure~\ref{fig:2obj_dist} shows reconstructions of two inclusions at several distance, given by the difference $\theta$ in the phase of the centroid (in polar coordinates).
The given data appears to allow distinction of objects very well, as long as they do not overlap. However, decreasing distance between them compromises the quality of reconstructions. 

\bigskip

\begin{figure}[htbp]
\begin{center}
\includegraphics[width=0.19\textwidth]{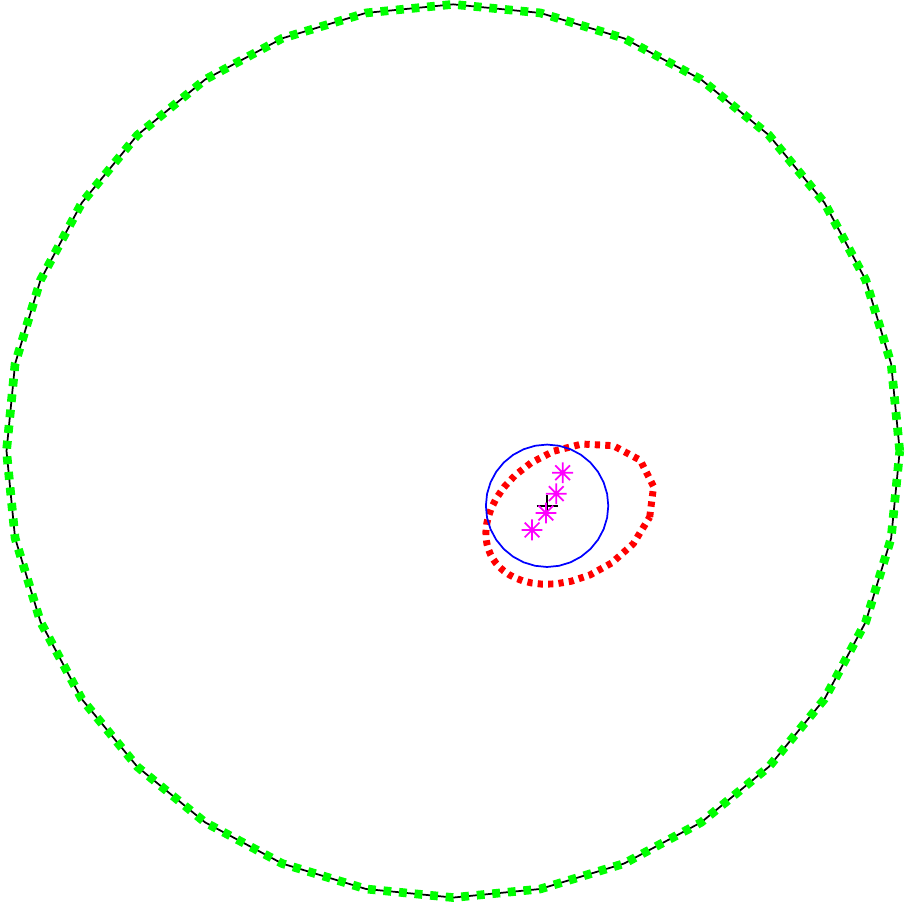}
\includegraphics[width=0.19\textwidth]{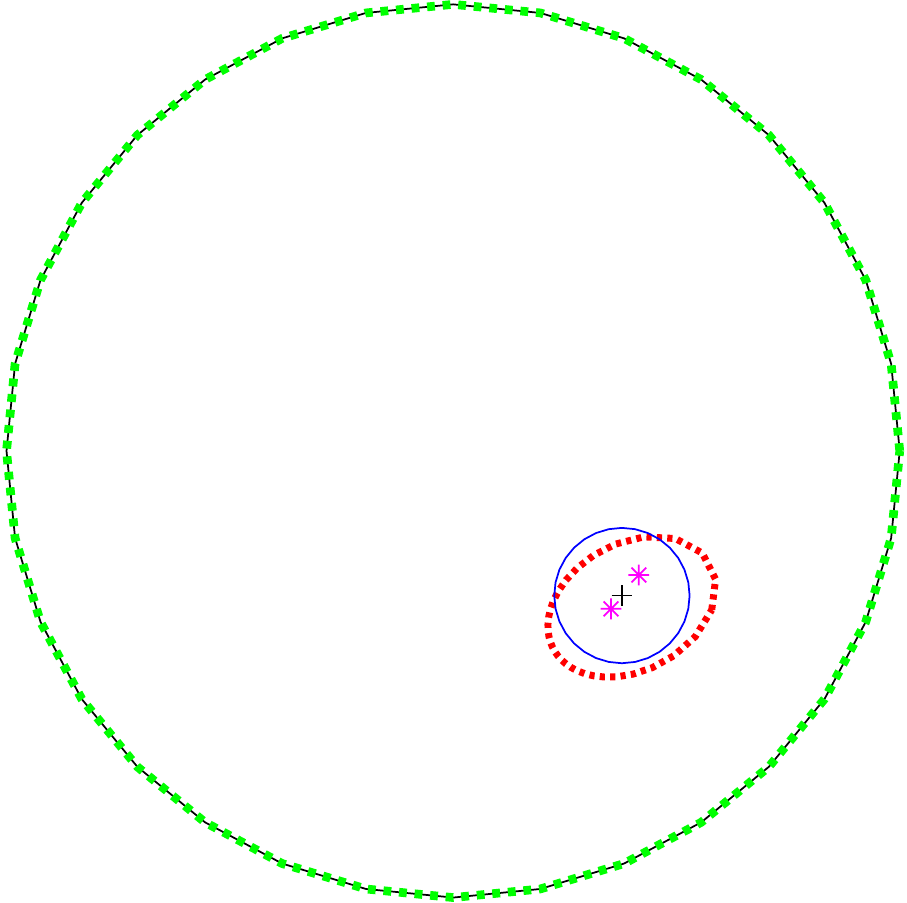}
\includegraphics[width=0.19\textwidth]{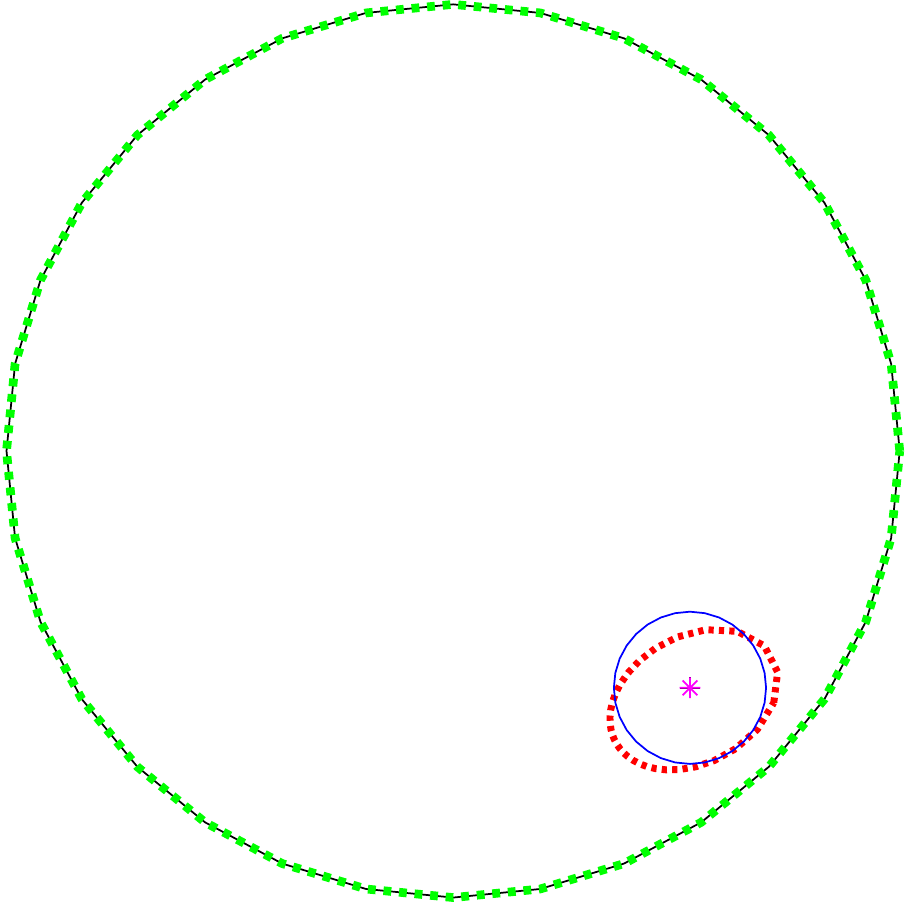}
\\[2ex]
\includegraphics[width=0.19\textwidth]{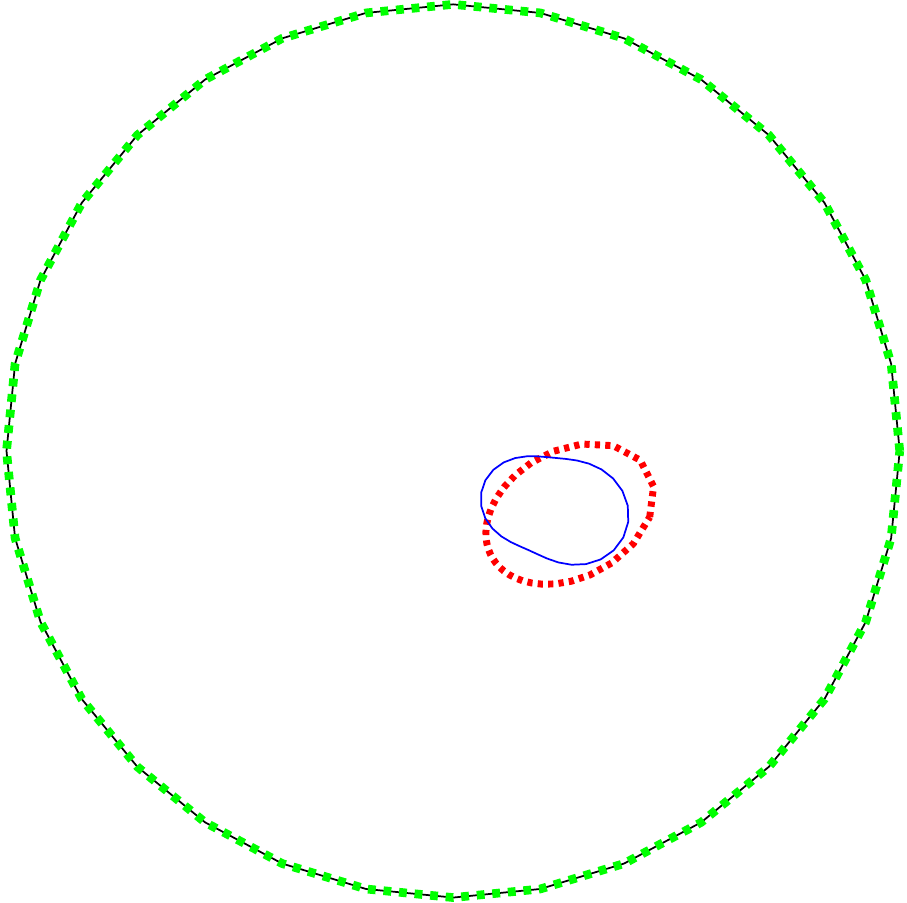}
\includegraphics[width=0.19\textwidth]{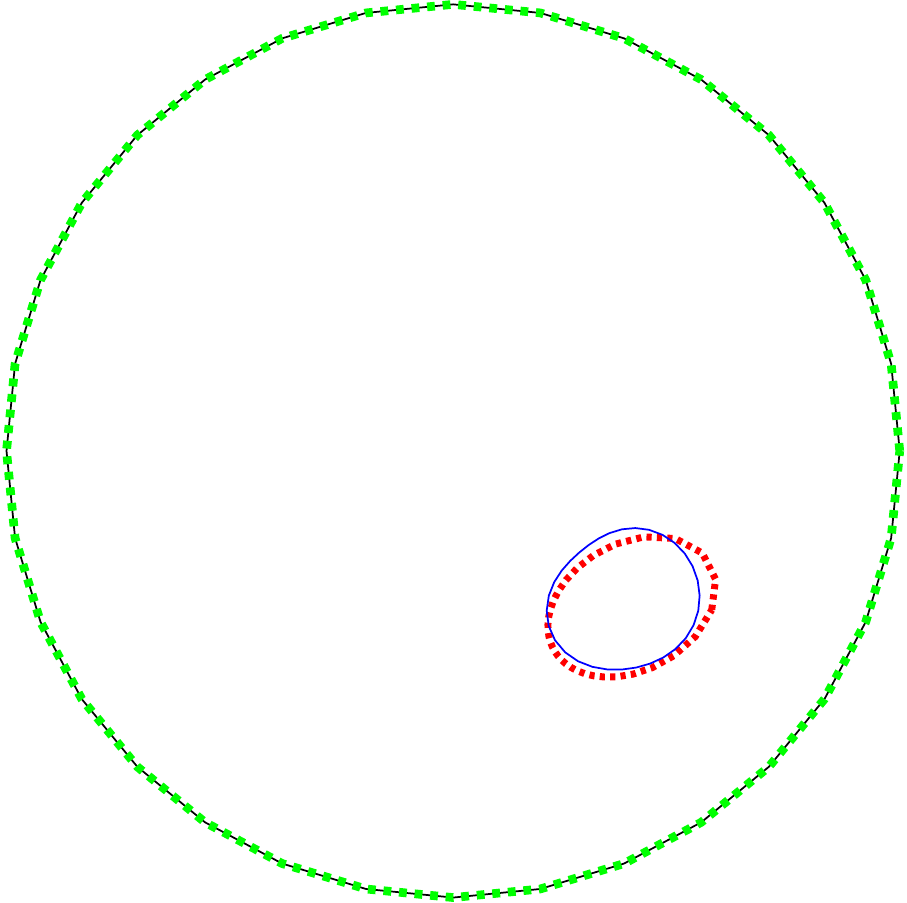}
\includegraphics[width=0.19\textwidth]{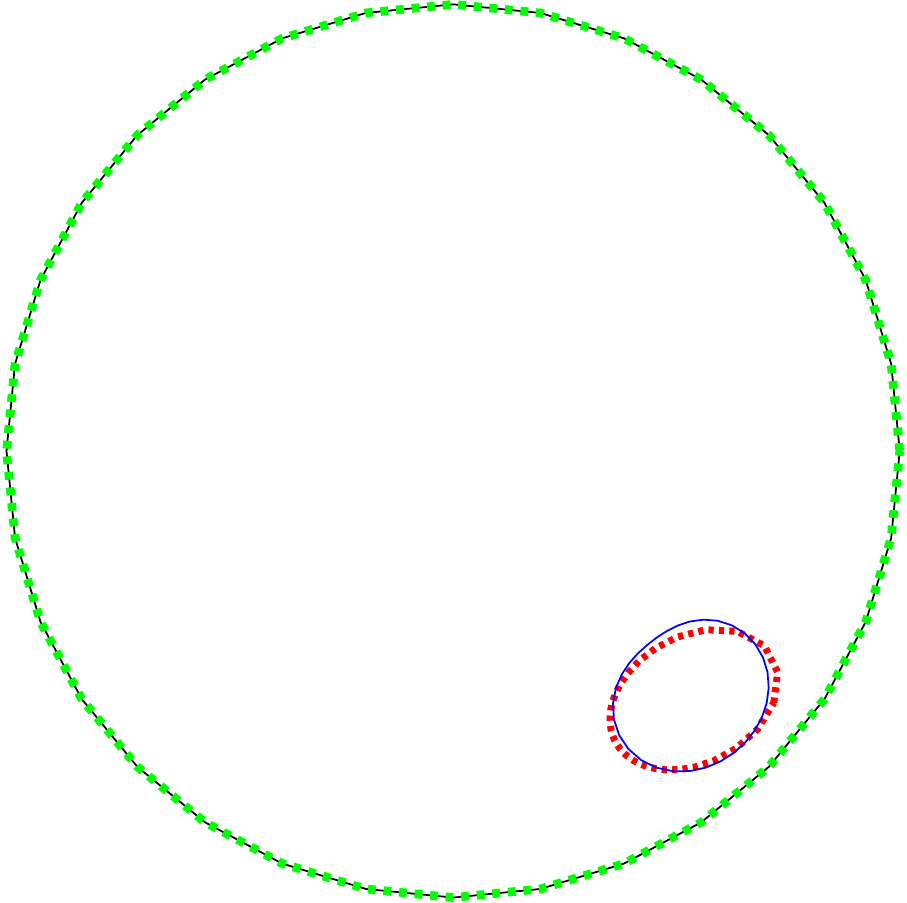}
\\[2ex]
(a) \hspace*{0.15\textwidth}
(b) \hspace*{0.15\textwidth}
(c) 
\caption{Reconstruction of one inclusion at different distances from the boundary;
top row: equivalent point sources and disks; bottom row: boundary curves from Newton's method
\label{fig:2obj_dist}}
\end{center}
\end{figure}
\paragraph{Varying distance to boundary:} 
Figure~\ref{fig:2obj_dist} shows reconstructions of one inclusion at several distances from the boundary. The relative error after application of Newton's method at $\kaptil=10$ was (a) 0.2963 (b) 0.1931 (c) 0.1434.
Also visually, it is obvious that closeness to the observation surface significantly improves the reconstruction quality.

\begin{figure}[htbp]
\begin{center}
\includegraphics[width=0.19\textwidth]{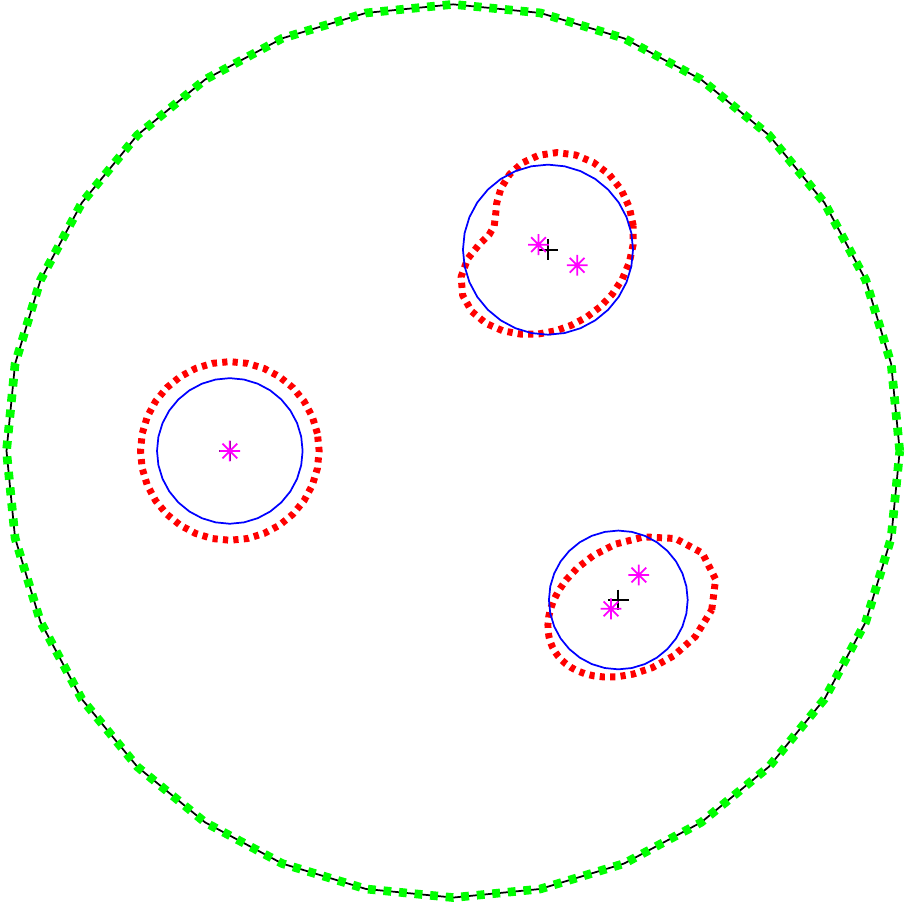}
\includegraphics[width=0.19\textwidth]{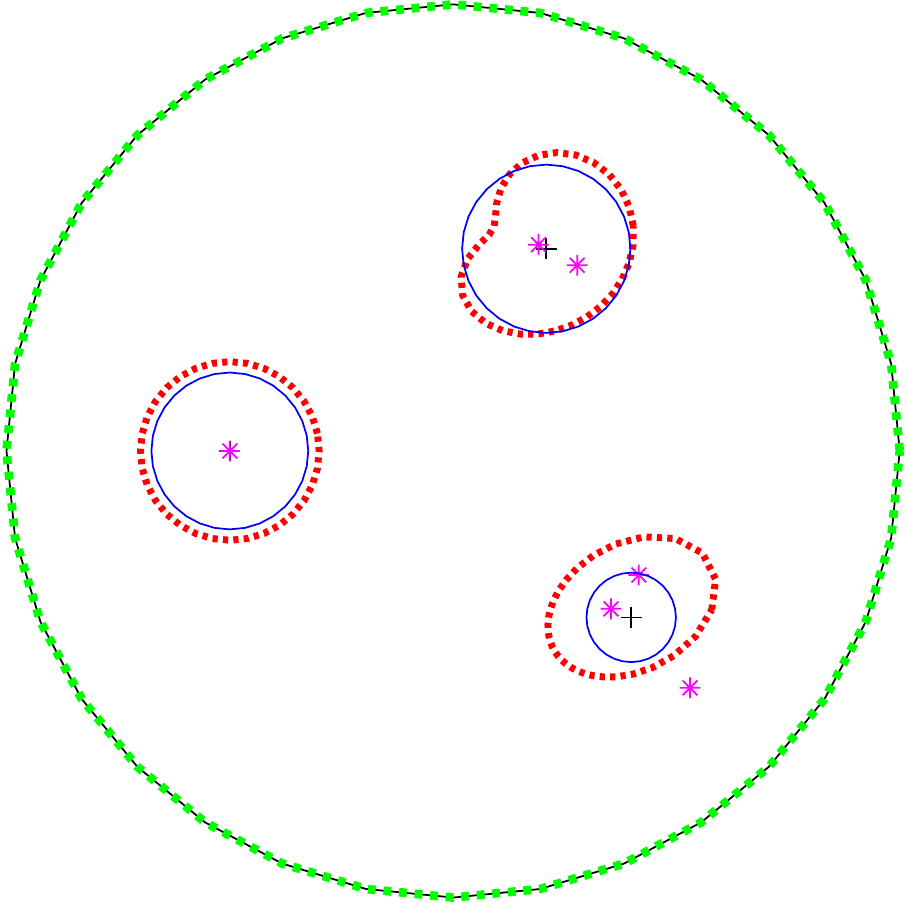}
\includegraphics[width=0.19\textwidth]{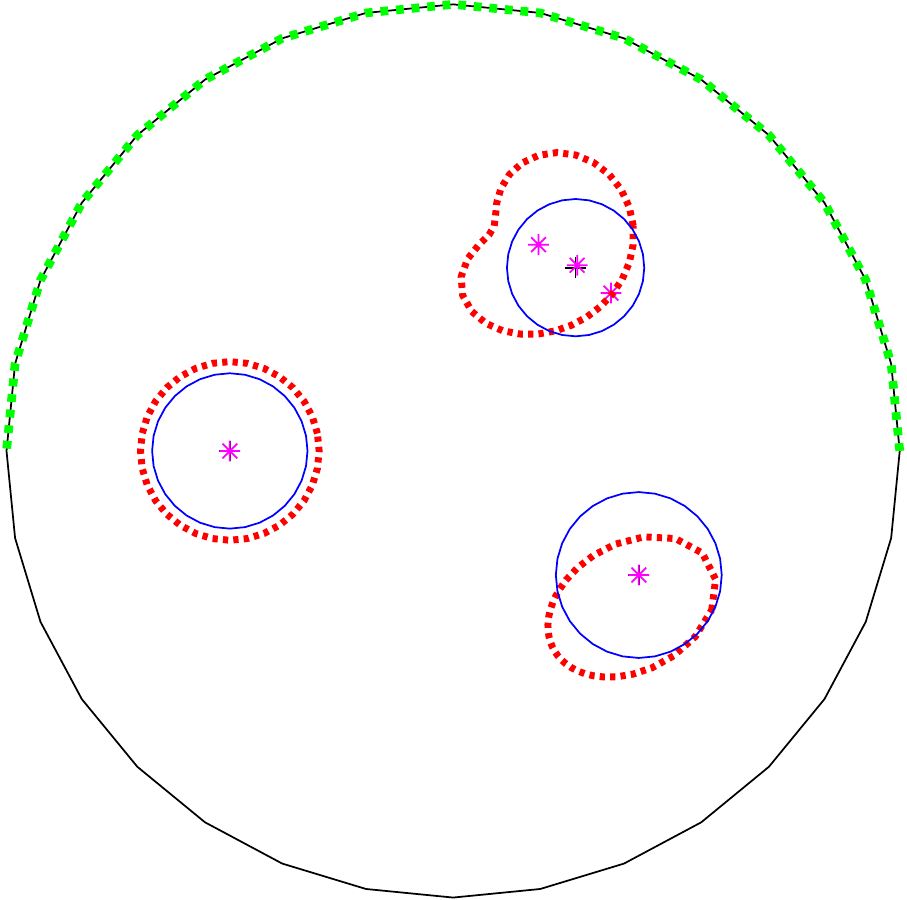}
\includegraphics[width=0.19\textwidth]{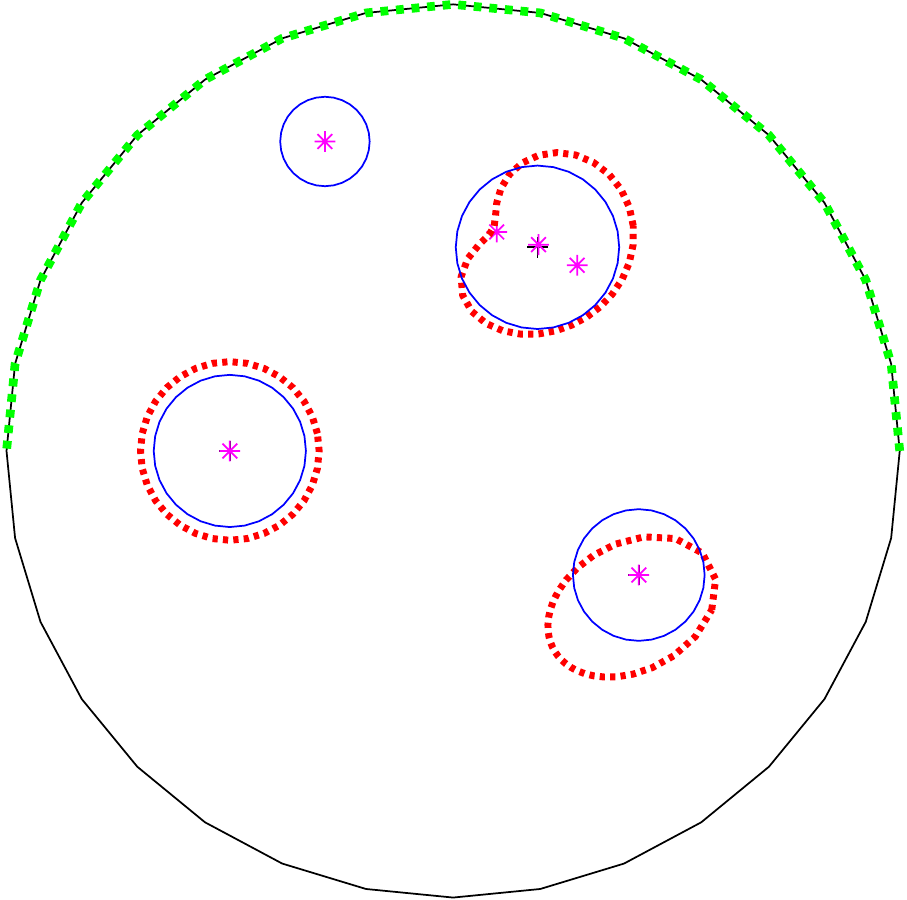}
\\[2ex]
\includegraphics[width=0.19\textwidth]{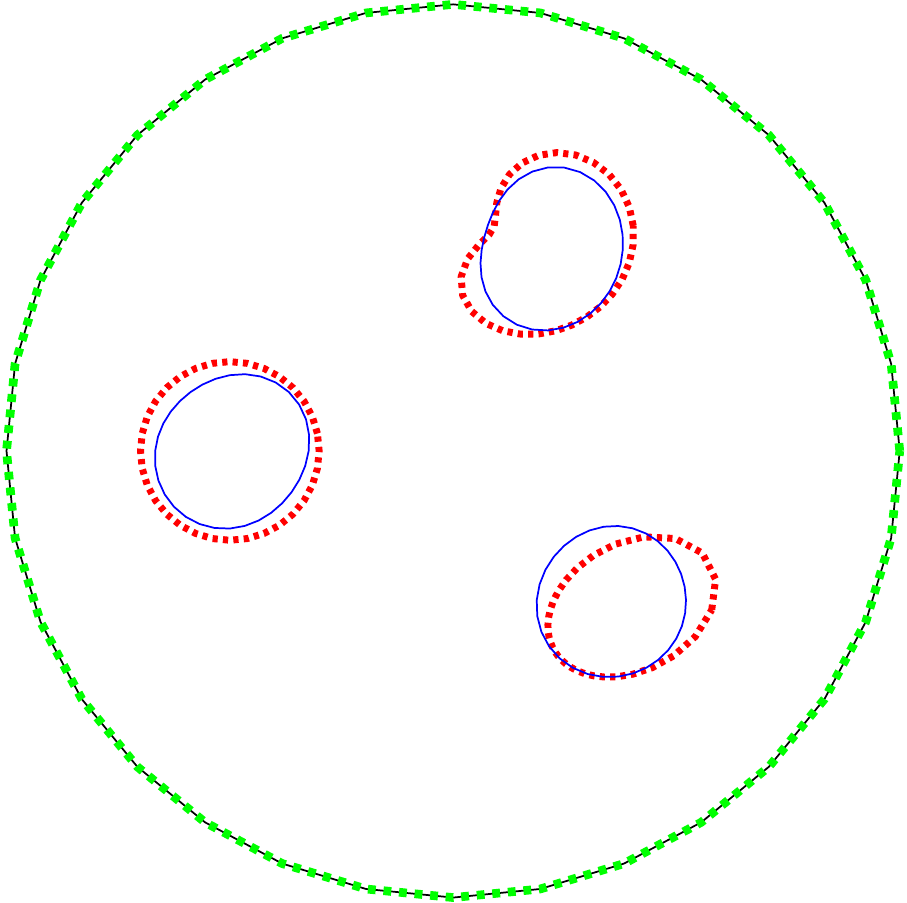}
\includegraphics[width=0.19\textwidth]{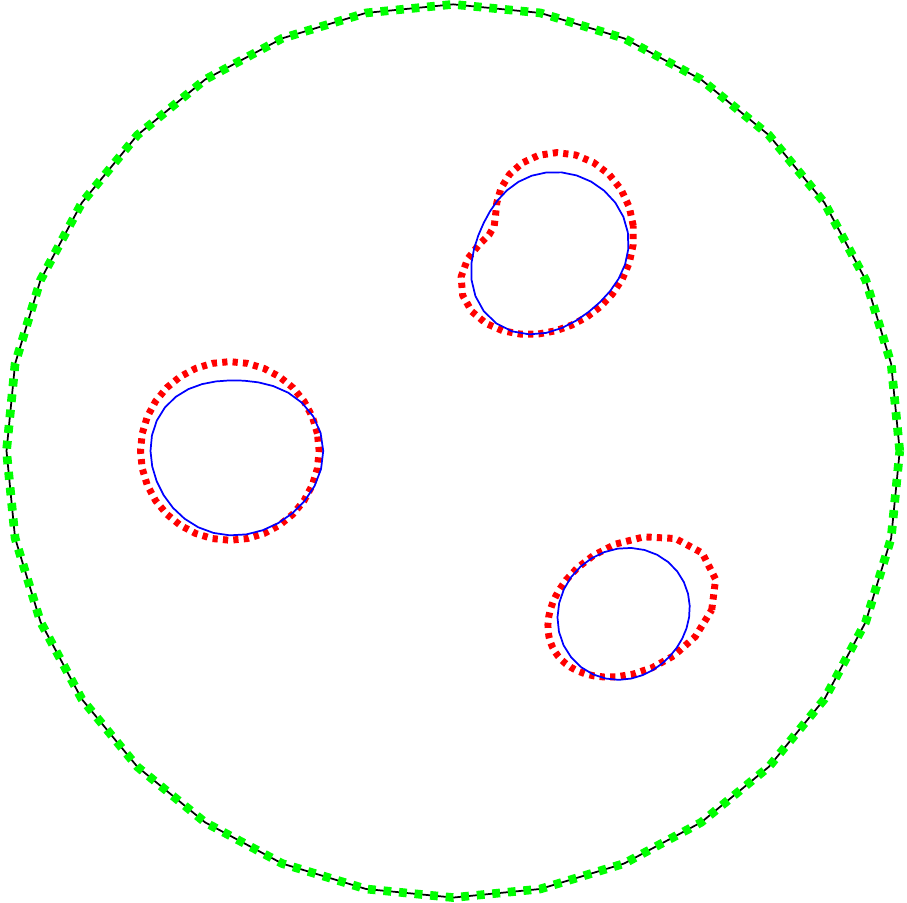}
\includegraphics[width=0.19\textwidth]{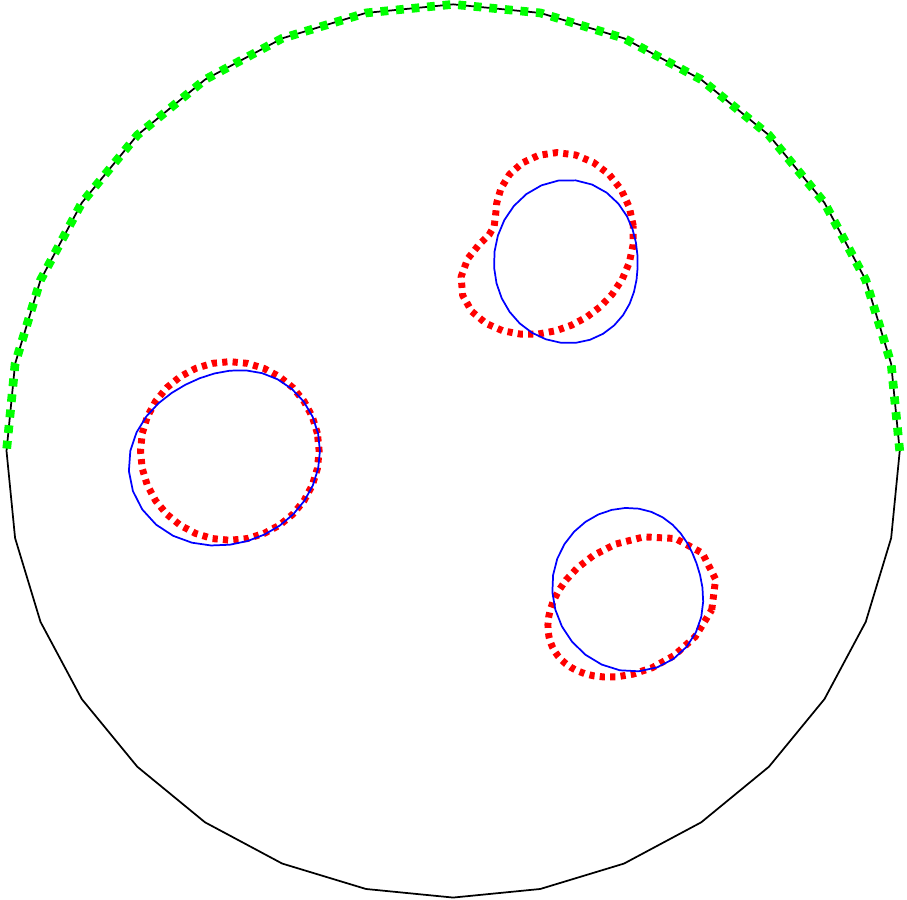}
\hspace*{0.19\textwidth}
\\[2ex]
(a) $\delta=2$ \% \hspace*{0.05\textwidth}
(b) $\delta=3$ \% \hspace*{0.05\textwidth}
(c) $\delta=2$ \% \hspace*{0.05\textwidth}
(d) $\delta=3$ \% \hspace*{-0.01\textwidth}\\
\hspace*{0.42\textwidth}
$\frac{\alpha}{2\pi}=0.5$ \hspace*{0.09\textwidth}
$\frac{\alpha}{2\pi}=0.5$\\
\caption{Reconstruction of three inclusions from noisy data;
top row: equivalent point sources and disks; bottom row: boundary curves from Newton's method
\label{fig:3obj_noise}}
\end{center}
\end{figure}
\paragraph{Reconstruction from noisy data:} 
Finally we study the impact of noise in the measurements on the reconstruction quality, see Figure~\ref{fig:3obj_noise} for the case of three objects. Regularisation is mainly achieved by the sparsity prior incorporated via the PDAP point source identification and this actually makes the process very stable with respect to perturbations in the measurements up to noise levels of about three per cent. 
Using partial data clearly impacts this robustness and thus only works with noise levels of two per cent or less.
 

\section{Convergence of Newton's method}
Similarly to the time domain setting, \cite{BB15} one can prove that the all-at-once formulation of this inverse problem (even with arbitrary $M\in\mathbb{N}\cup\{\infty\}$) satisfies a range invariance condition, which, together with a linearised uniqueness result, enables to prove convergence of a regularised frozen Newton method. 

We write the inverse problem of reconstructing $\eta$ in \eqref{multiharmonic_nonlinear} as 
a nonlinear operator equation
\begin{equation}\label{eqn:Fm0}
\begin{aligned}
&G_m(\eta,\hat{p}) = h_m \quad m\in\{1,\ldots,M\}	\text{ with } \hat{p}=(\hat{p}_1,\ldots,\hat{p}_M) \\
&C_m \hat{p}_m = y_m  \quad m\in\{1,\ldots,M\}
\end{aligned}
\end{equation}
for the model operators $G_m:Q\times V^M\to W$ (including the case $M=\infty$ with $\ell^2(\mathbb{N};V)$ in place of $V^M$), $h_1=\hat{h}$, $h_m=0$ for $m\geq2$ and the observation operators $C_m\in L(V,Y)$. 
Here $Q$, $V$, $Y$ are the parameter, state, and data spaces.

The components $G_m$ of the model part of the forward operator have the particular structure
\begin{equation}\label{eqn:FDB}
G_m(\eta,\hat{p}) = D_m \hat{p}_m + B_m(\hat{p}) \eta
\end{equation}
with $D_m\in L(V,W)$ and $B_m(\hat{p})\in L(Q, W)$ linear for each $\hat{p}\in V^M$ but depending nonlinearly on $\hat{p}$.
(This is different from \cite{rangeinvar}, where we considered a sum of linear operators $B_m(\hat{p})$ in a single model equation rather than a system of model equations.) 
More concretely, in our setting with the operators defined by 
\begin{equation}\label{eqn:calADM}
\begin{aligned}
&\mathcal{A}u = \Bigl(v\mapsto \int_\Omega\nabla u\cdot\nabla v\, dx +\gamma\int_{\partial\Omega} u\, v\, ds\Bigr)\\
&\mathcal{D}u = \Bigl(v\mapsto b\int_\Omega\nabla u\cdot\nabla v\, dx +(c^2\beta+b\gamma)\int_{\partial\Omega} u\, v\, ds\Bigr)\,, \\
&\mathcal{M}u = \Bigl(v\mapsto \int_\Omega u\,v\, dx +\beta b\int_{\partial\Omega} u\, v\, ds\Bigr)
\end{aligned}
\end{equation}
we take 
\begin{equation}\label{eqn:DBQVW}
\begin{aligned}
&D_m  = 
-m^2\omega^2\mathcal{M}+c^2\mathcal{A}+\imath\,m\omega \, \mathcal{D},
\qquad C_m = \textup{tr}_\Sigma,
\\
&
B_m(\hat{p})(x) = m^2\omega^2\tilde{B}_m(\hat{p}(x))\\
&\tilde{B}_m(\vec{c})= \begin{cases}
\frac14\sum_{\ell=1}^{m-1} c_\ell c_{m-\ell} 
+ \frac12\sum_{k=m+2:2}^\infty\overline{c_{\frac{k-m}{2}}} c_{\frac{k+m}{2}}
& M=\infty \text{ (a)}\\
\frac14\sum_{\ell=1}^{m-1} c_\ell c_{m-\ell} 
& M\in\mathbb{N}\cup\{\infty\} \text{ (b)}\\
\end{cases} \quad \vec{c}\in\mathbb{C}^M\\
&Q = L^2(\Omega), \quad V= H^2(\Omega), \quad W = L^2(\Omega), \quad Y = L^2(\Omega), 
\end{aligned}
\end{equation}
where the first sum over $\ell$ is empty in the case $m=1$.
Here $B_m(\hat{p}):L^2(\Omega)\to L^2(\Omega)$ is to be understood as a multiplication operator and boundedness $B_m(\hat{p})\in L(L^2(\Omega),L^2(\Omega))$ follows from the fact that $H^2(\Omega)$ is continuously embedded in $L^\infty(\Omega)$ and therefore the functions $p_m$ as well as their products are in $L^\infty(\Omega)$. 
Differentiability of the $B_m$ mappings follows from their polynomial (in fact, quadratic) structure in our particular setting.

We consider both the case (a) that gives full equivalence to the Westervelt equation \eqref{Westervelt} and the simplifications (b) used in our numerical tests. 

\medskip

The abstract structure \eqref{eqn:Fm0}, \eqref{eqn:FDB} together with an extension of the dependency of $\eta$ to $\vec{\eta}=(\eta_m)_{m\in\{1,\ldots,M\}}\subseteq Q^M$ allows one to more generally establish the differential range invariance relation 
\begin{equation}\label{rangeinvar_diff}
\textup{ for all } (\vec{\eta},\hat{p})\in U \, \exists r(\vec{\eta},\hat{p})\in Q^M\times V^M\,: \  F(\vec{\eta},\hat{p})-F(\vec{\eta}_0,\hat{p}_0)=F'(\vec{\eta}_0,\hat{p}_0)r(\vec{\eta},\hat{p}),
\end{equation}
for 
\begin{equation}\label{F}
\begin{aligned}
&F=(G_m,C_m)_{m\in\{1,\ldots,M\}}, \quad \hat{p}=(\hat{p}_m)_{m\in\{1,\ldots,M\}}, \\ 
&r(\vec{\eta},\hat{p})=(r^{\vec{\eta}}_m(\vec{\eta},\hat{p}),r^{\hat{p}}_m(\vec{\eta},\hat{p}))_{m\in\{1,\ldots,M\}}.
\end{aligned}
\end{equation}
Indeed, with 
\[
G_m'(\vec{\eta}_0,\hat{p}_0)(\uldeta,\uldph)=
D_m \uldph_m + \sum_{n=1}^M\frac{\partial B_m}{\partial \hat{p}_n}(\hat{p}_0)\uldph_n\, \vec{\eta}_{0,m} + B_m(\hat{p}_0) \uldeta_m
\]
and 
\[
\begin{aligned}
&r^{\hat{p}}_m(\vec{\eta},\hat{p})=\hat{p}_m-\hat{p}_{0,m}\\
&r^{\vec{\eta}}_m(\vec{\eta},\hat{p})=\eta_m-\eta_{0,m}+
B_m(\hat{p}_0)^{-1}\Bigl(
\bigl(B_m(\hat{p})-B_m(\hat{p}_0)\bigr)\eta_m
-\sum_{n=1}^M\frac{\partial B_m}{\partial \hat{p}_n}(\hat{p}_0)(\hat{p}_m-\hat{p}_{0,m})\, \eta_{0,m}\Bigr)
\end{aligned}
\]
we obtain \eqref{rangeinvar_diff}.
To this end, we assume that $p_0$ is chosen such that for each $m\in\{1,\ldots,M\}$, the operator $B_m(\hat{p}_0):Q\to W$ is an isomorphism. 
Moroever, $r$ is close to the identity in the sense that 
\[
\begin{aligned}
&\|r(\vec{\eta},\hat{p})-((\vec{\eta},\hat{p})-(\vec{\eta}_0,\hat{p}_0))\|_{Q^M\times V^M}\\
&=\Bigl\|B_m(\hat{p}_0)^{-1}\Bigl(
\bigl(B_m(\hat{p})-B_m(\hat{p}_0)\bigr)(\eta_m-\eta_{0,m})\\
&\hspace*{3cm}+\bigl(B_m(\hat{p})-B_m(\hat{p}_0)-\sum_{n=1}^M\frac{\partial B_m}{\partial \hat{p}_n}(\hat{p}_0)(\hat{p}_m-\hat{p}_{0,m})\bigr)\, \eta_{0,m}\Bigr)\Bigr\|_{Q^M\times V^M}\\
&\leq C \|\hat{p}-\hat{p}_0\|_{V^M}\bigl(\|\eta-\eta_0\|_{Q^M}+\|\hat{p}-\hat{p}_0\|_{V^M}\bigr),
\end{aligned}
\]
which implies
\begin{equation}\label{rid}
\|r(\vec{\eta},\hat{p})-((\vec{\eta},\hat{p})-(\vec{\eta}_0,\hat{p}_0))\|_{Q^M\times V^M}
\leq c \|(\vec{\eta},\hat{p})-(\vec{\eta}_0,\hat{p}_0)\|_{Q^M\times V^M}
\end{equation}
for $c\in(0,1)$ in a sufficiently small neighborhood $U$ of $(\vec{\eta}_0,\hat{p}_0)$.

\medskip

Since the artificial dependence of $\vec{\eta}$ on $m$ counteracts uniqueness, we penalise it by a term $P\vec{\eta}\in Q^M$
\[
(P\vec{\eta})_m = \eta_m - \frac{\sum_{n=1}^M n^{-2}\,\eta_n}{\sum_{n=1}^M n^{-2}}, 
\]
where the weights $n^{-2}$ in the $\ell^2$ projection are introduced in order to enforce convergence in case $M=\infty$.
Note that the $n$ independent target $(\eta,\eta,\ldots )$ is clearly not contained in $\ell^2(\mathbb{N}; Q)$ but in the weighed space $\ell^2_w(\mathbb{N}; Q)$ with weights $w_n=n^{-2}$.
We here first of all aim at finding a general $\eta\in Q=L^2(\Omega)$.
In case we want to reconstruct a piecewise constant coefficient $\eta$, we can achieve this by, e.g., adding a total variation term to $P$. 

This penalisation together with condition \eqref{rangeinvar_diff} allows us to rewrite the inverse problem \eqref{eqn:Fm0} as a combination of an ill-posed linear and a well-posed nonlinear problem
\begin{equation}\label{FP}
\begin{aligned}
&F'(\vec{\eta}_0,\hat{p}_0)\hat{r}=h-F(\vec{\eta}_0,\hat{p}_0)\\
&r(\vec{\eta},\hat{p})=\hat{r}\\
&P\vec{\eta}=0
\end{aligned}
\end{equation}
for the unknowns $(\hat{r},\vec{\eta},\hat{p})\in Q^M\times Q^M\times V^M$ (or in $\ell^2_w(\mathbb{N};Q)\times\ell^2_w(\mathbb{N}; Q)\times\ell^2(\mathbb{N}; V)$ in case $M=\infty$). 
Here $(\vec{\eta}_0,\hat{p}_0)\in Q^M\times V^M$ is fixed and in \eqref{rangeinvar_diff} $U\subseteq Q^M\times V^M$ is a neighborhood of $(\vec{\eta}_0,\hat{p}_0)$.

The following regularised frozen Newton method can then be shown to converge.
\begin{equation}\label{frozenNewtonHilbert}
x_{n+1}^\delta \in \mbox{argmin}_{x\in U}
\|F'(x_0)(x-x_n^\delta)+F(x_n^\delta)-h^\delta\|_Y^2+\alpha_n\|\vec{\eta}-\vec{\eta}_0\|_{Q^M}^2+\|P\vec{\eta}\|_{Q^M}^2.
\end{equation}
where $h^\delta\approx h$ is the noisy data, $\alpha_n\to0$ as $n\to\infty$, (e.g. $\alpha_n=\alpha_0 q^n$ for some $q\in(0,1)$), and we abbreviate $x=(\vec{\eta},\hat{p})$.

An essential ingredient of the convergence proof is verification of the fact that the intersection of the nullspaces of $F'(x_0)$ and of $P$ is trivial \cite[Theorem 2]{rangeinvar}.
For this purpose, 
we require the following geometric condition on the observation manifold $\Sigma$
\begin{equation}\label{eqn:ass_inj_Sigma_rem}
\textup{ for all } j\in\mathbb{N}\ : \quad 
\left(\sum_{k\in K^j} b_k \varphi_j^k(x) = 0 \ \mbox{ for all }x\in\Sigma\right)
\ \Longrightarrow \ \left(b_k=0 \mbox{ for all }k\in K^j\right)
\end{equation}
in terms of the eigensystem $(\varphi_j^k,\lambda_j)_{j\in\mathbb{N},k\in K^j}$ of the selfadjoint positive operator $\mathcal{A}$ defined by \eqref{eqn:calADM}.
This means that the eigenfunctions should preserve their linear independence when restricted to the observation manifold and trivially holds in 1-d, where $\# K^j=1$ for all $j\in\mathbb{N}$.

We will assume that the operators $\mathcal{A}$, $\mathcal{D}$, $\mathcal{M}$ have the same $H$-orthonormal eigenfunctions $\varphi_j^k$ with the eigenvalues $\mu_j$ of $\mathcal{M}$ and $\rho_j$ of $\mathcal{D}$ satisfying  
\begin{equation}\label{eqn:rho_lambda_mu}
\Bigl(\frac{\rho_j}{\lambda_j}=\frac{\rho_\ell}{\lambda_\ell} \textup{ and } 
\frac{\mu_j}{\lambda_j^2}=\frac{\mu_\ell}{\lambda_\ell^2}\Bigr)
\ \Rightarrow \ j=\ell.
\end{equation} 
This is the case, e.g., if $\beta=0$, where $\mathcal{M}$ is the identity and $\mathcal{D}=b\mathcal{A}$, $H=L^2(\Omega)$. Condition \eqref{eqn:rho_lambda_mu} is needed to prove the following linear independence result that will play a role in the linearized uniqueness result Theorem~\ref{thm:linearizeduniqueness}.
Its proof can be found in the appendix.

\begin{lemma}\label{lem:linindep}
Let $(\mu_j)_{j\in\mathbb{N}}$, $(\lambda_j)_{j\in\mathbb{N}}$, $(\rho_j)_{j\in\mathbb{N}}$ $\subseteq\mathbb{C}$ sequence of distinct numbers such that \eqref{eqn:rho_lambda_mu} holds.\\
Then 
\[
\left(\mbox{ for all }m\in\mathbb{N}\, : \ 0=\sum_{j=1}^\infty 
\frac{m^2}{-m^2\omega^2\mu_j+c^2\lambda_j+\imath m \omega \rho_j} c_j \right)
\ \Longrightarrow \ \left(c_j=0 \mbox{ for all }j\in\mathbb{N}\right)
\]
\end{lemma}

We are now in the position to prove uniqueness for the linearized problem, which, besides being of interest on its own, is also an essestial ingredient to the convergence proof of Newton's method.

\begin{theorem}\label{thm:linearizeduniqueness}
For \eqref{F}, \eqref{eqn:FDB}, \eqref{eqn:DBQVW}, with $M=\infty$ and $\eta$ independent of $m$ (that is, $P\vec{\eta}=0$), $\hat{p}_0$ chosen such that $\hat{p}_{0,m}(x)= \phi(x)\,\psi_m$ for some $\phi\in H^2(\Omega)$, $\phi\not=0$ almost everywhere in $\Omega$, 
$\psi_m\in\mathbb{C}$, $f_m:=\tilde{B}_m(\vec{\psi})\in\mathbb{C}\setminus\{0\}$ for all $m\in\mathbb{N}$.
\\
Then under the linear independence condition \eqref{eqn:ass_inj_Sigma_rem}, with $\mathcal{A}$, $\mathcal{D}$, $\mathcal{M}$ simultaneously diagonalisable with \eqref{eqn:rho_lambda_mu}, the linearisation $F'(\eta_0,\hat{p}_0)$ at $\eta_0=0$ is injective. 
\end{theorem}
\begin{proof}
Using the operators $\mathcal{A}$, $\mathcal{D}$, $\mathcal{M}$ as in \eqref{eqn:calADM} we can write the condition $F'(\eta_0,\hat{p}_0)(\uldeta,\uldp)$ for $\eta_0=0$, $\hat{p}_{0,m}(x)= \phi(x)\,\psi_m$,  $f_m=\tilde{B}_m(\vec{\psi})$ as 
\begin{equation}\label{Fprime0}
[-m^2\omega^2\mathcal{M}+c^2\mathcal{A}+\imath\,m\omega \, \mathcal{D}]\uldp_m+m^2\omega^2f_m\phi\,\uldeta =0,
\textup { and } \textup{tr}_\Sigma \uldp_m=0 \textup{ for all } m\in\mathbb{N}.
\end{equation}
Using the diagonalisation by means of the eigenfunctions $(\varphi_j^k)_{j\in\mathbb{N},k\in K^j}$, by taking the $H$ inner product of \eqref{Fprime0} with $\varphi_j^k$, relying on $\uldp_m=\sum_{j=1}^\infty\sum_{k\in K^j}\langle \uldp_m,\varphi_j^k\rangle_H \varphi_j^k$
and setting $a_j^k=\langle\uldeta\,\phi,\varphi_j^k\rangle_H$
we can rewrite this as 
\[
m^2\omega^2f_m \sum_{j=1}^\infty \frac{1}{-m^2\omega^2\mu_j+c^2\lambda_j+\imath\,m\omega \, \rho_j}\sum_{k\in K^j} a_j^k \varphi_j^k(x_0) =0 \textup{ for all } x_0\in\Sigma, \ m\in\mathbb{N}.
\]
Since the entries 
$\frac{1}{-m^2\omega^2\mu_j+c^2\lambda_j+\imath m \omega \rho_j}$
define an infinite generalised Hankel matrix which is therefore nonsingular (see Lemma~\ref{lem:linindep}),
this implies 
\[
0=\sum_{k\in K^j} a_j^k\, \varphi_j^k(x_0) \quad \textup{ for all } j\in \mathbb{N}, \ x_0\in\Sigma.
\]
Using \eqref{eqn:ass_inj_Sigma_rem}, we conclude $a_j^k=0$ for all $j\in\mathbb{N}$, $k\in K^j$ and thus $\uldeta=0$.
Returning to the first equation in \eqref{Fprime0} with $\uldeta=0$, due to uniqueness of the solution to this linear homogeneous PDE with homogeneous boundary conditions, we also have $\uldp=0$.
\end{proof} 

According to \cite[Theorem 2]{rangeinvar}, we obtain the following
\begin{theorem}\label{thm:convfrozenNewton}
Let  $x^\dagger=(\vec{\eta}^\dagger,\hat{p}^\dagger)$ be a solution to \eqref{FP} and let for the noise level $\delta\geq\|y^\delta-y\|_Y$ the stopping index $n_*=n_*(\delta)$ be chosen such that 
\begin{equation}\label{nstar}
n_*(\delta)\to0, \quad \delta\sum_{j=0}^{n_*(\delta)-1}c^j\alpha_{n_*(\delta)-j-1}^{-1/2} \to 0 \qquad \textup{ as }\delta\to0
\end{equation}
with $c$ as in \eqref{rid}.
Moreover, let the assumptions of Theorem~\ref{thm:linearizeduniqueness} be satisfied with $B_m(\hat{p}_0)$ as in \eqref{eqn:DBQVW} being an isomorphism from $L^2(\Omega)$ into itself for all $m\in\mathbb{N}$.

Then there exists $\rho>0$ sufficiently small such that for $x_0\in\mathcal{B}_\rho(x^\dagger)\subseteq U$ the iterates $(x_n^\delta)_{n\in\{1,\ldots,n_*(\delta)\}}$ are well-defined by \eqref{frozenNewtonHilbert}, remain in $\mathcal{B}_\rho(x^\dagger)$ and converge in $Q^M\times V^M$, $\|x_{n_*(\delta)}^\delta-x^\dagger\|_{Q^M\times V^M}\to0$ as $\delta\to0$. In the noise free case $\delta=0$, $n_*(\delta)=\infty$ we have $\|x_n-x^\dagger\|_{Q^M\times V^M}\to0$ as $n\to\infty$.
\end{theorem}

\section*{Appendix}
Proof of Lemma~\ref{lem:linindep}:\\
With 
$w_j(t):=-\mu_j\omega^2+c^2\lambda_j\, t^2+\imath \omega \rho_j\, t$, 
the premise of the lemma reads as
\[
\mbox{ for all }t\in\{\tfrac{1}{m}\, : \, m\in \mathbb{N}\}\, : \quad 
0=\sum_{j=1}^\infty\tfrac{1}{w_j(t)} \, c_j.
\]
Thus, after multiplication with $\prod_{\ell\in\mathbb{N}}w_\ell(t)$ and with $W^{\vec{c}}(t):=\sum_{j=1}^\infty\prod_{\ell\not=j}w_\ell(t)\, c_j$ we get 
\[
\mbox{ for all }t\in\{\tfrac{1}{m}\, : \, m\in \mathbb{N}\}\, : \quad 
0=W^{\vec{c}}(t).
\]
Since $W^{\vec{c}}$ is analytic, this implies that $W^{\vec{c}}\equiv0$ on all of $\mathbb{C}$. Choosing 
$t_{k\pm}=-\frac{\imath\omega}{2c^2}\frac{\rho_k\mp\sqrt{\rho_k^2-\mu_k}}{\lambda_k}$ as the roots of $w_k$, we obtain 
\begin{equation}\label{eqn:wellck0}
\mbox{ for all }k\in\mathbb{N}\,: \quad \prod_{\ell\not=k}w_\ell(t_{k\pm})\, c_k = 0
\end{equation}
A small side calculation yields that under condition \eqref{eqn:rho_lambda_mu},
the roots of the functions $w_j$ are distinct for different $j$:
\[
\begin{aligned}
\Bigl(t_{j+}=t_{\ell+}\textup{ and }t_{j-}=t_{\ell-}\Bigr) \ \Rightarrow \ 
\Bigl(t_{j+}+t_{j-}=t_{\ell+}+t_{\ell-}\textup{ and }t_{j+}+t_{j-}=t_{\ell+}+t_{\ell-}\Bigr) \\ \Rightarrow \ 
\Bigl(\frac{\rho_j}{\lambda_j}=\frac{\rho_\ell}{\lambda_\ell} \textup{ and } 
\frac{\mu_j}{\lambda_j^2}=\frac{\mu_\ell}{\lambda_\ell^2}\Bigr),
\end{aligned}
\]
which by \eqref{eqn:rho_lambda_mu} implies $j=\ell$.\\
Hence, $\prod_{\ell\not=k}w_\ell(t_{k\pm})\not=0$ and from \eqref{eqn:wellck0} we conclude that $c_k=0$ for all $k\in\mathbb{N}$.

\section*{Acknowledgment}
The work of the first author was supported by the Austrian Science Fund through grant P36318; the second author was supported in part by the National Science Foundation through award DMS -2111020.


\end{document}